\documentclass{amsart}
\usepackage{amsmath,amsthm}
\usepackage{amsfonts,amssymb}

\usepackage{enumerate}

\newtheorem{thm}{Theorem}[section]
\newtheorem{cor}[thm]{Corollary}
\newtheorem{lem}[thm]{Lemma}
\newtheorem{prop}[thm]{Proposition}

\newtheorem{defn}[thm]{Definition}

\theoremstyle{remark}
\newtheorem{rem}{Remark}[section]

\newcommand{\mR}{\mathbb{R}}

\newcommand{\mN}{\mathbb{N}}
\newcommand{\mE}{\mathbb{E}}

\newcommand{\mS}{\mathbb{S}}
\newcommand{\mP}{\mathbb{P}}

\newcommand{\cM}{\mathcal{M}}
\newcommand{\cH}{\mathcal{H}}

\newcommand{\cF}{\mathcal{F}}
\newcommand{\cP}{\mathcal{P}}
\newcommand{\cC}{\mathcal{C}}

\newcommand{\cS}{\mathcal{S}}
\newcommand{\cG}{\mathcal{G}}

\newcommand{\ux}{\underline{x}}

\newcommand{\uy}{\underline{y}}

\newcommand{\upx}{\partial_{\underline{x}}}
\newcommand{\upy}{\partial_{\underline{y}}}

\newcommand{\pI}{\partial_{x_i}}

 \def\a{{\alpha}} 
 \def\b{{\beta}}
 \def\g{{\gamma}}
 
 \def\t{{\theta}}
 \def\l{{\lambda}}

 \def\s{{\sigma}}
 \def\la{{\langle}}
 \def\ra{{\rangle}}

 \def\CF{{\mathcal F}}

 \def\NN{{\mathbb N}}

 \def\RR{{\mathbb R}}

\newcommand{\wt}{\widetilde}
\newcommand{\wh}{\widehat}

\begin{document}
 
\title
{On the Clifford-Fourier transform}
\author{Hendrik De Bie}
\address{Department of Mathematical Analysis\\
Ghent University\\ Krijgslaan 281, 9000 Gent\\ Belgium.}
\email{Hendrik.DeBie@UGent.be}
\author{Yuan Xu}
\address{Department of Mathematics\\ University of Oregon\\
    Eugene, Oregon 97403-1222.}\email{yuan@math.uoregon.edu}

\date{\today}
\thanks{H. De Bie is a Postdoctoral Fellow of the Research Foundation - Flanders (FWO) and Courtesy Research Associate at the University of Oregon.} 

\begin{abstract}
For functions that take values in the Clifford algebra, we study the Clifford-Fourier transform on $\RR^m$ defined with a kernel
function $K(x,y) :=  e^{ \frac{i \pi }{2} \Gamma_{\uy} }e^{-i \langle \ux,\uy\rangle}$, replacing the kernel $e^{i  \langle \ux,\uy\rangle}$ of the ordinary Fourier transform, where $\Gamma_{\uy} := - \sum_{j<k} e_{j}e_{k} (y_{j} \partial_{y_{k}} - y_{k}\partial_{y_{j}})$. An explicit formula of $K(x,y)$ is derived, which can be further simplified to a finite sum of Bessel functions when $m$ is even. The closed formula of the kernel allows us to study the Clifford-Fourier transform and prove the inversion formula, for which a generalized translation operator and a convolution are defined and used. 
\end{abstract}

\maketitle

\tableofcontents

\section{Introduction}
\setcounter{equation}{0}

In harmonic analysis in, say, $\mR^{m}$, a profound role is played by the Lie algebra $\mathfrak{sl}_{2}$ generated by the Laplace operator $\Delta$ and the norm squared of a vector $|\ux|^{2}$ (see e.g. \cite{MR1151617}). As an example, the classical Fourier transform given by
\begin{equation}
\cF(f)(y) = (2 \pi)^{-\frac{m}{2}} \int_{\mR^{m}} e^{-i \la \ux, \uy \ra} f(x) dx, \quad \langle \ux,\uy \rangle = \sum_{i=1}^{m}x_{i} y_{i}
\label{classFourier}
\end{equation}
can be equivalently represented by the operator exponential that contains the generators of $\mathfrak{sl}_{2}$, 
\begin{equation}
\cF = e^{ \frac{i \pi m}{4}} e^{\frac{i \pi}{4}(\Delta - |\ux|^{2})},
\label{expclassFourier}
\end{equation}
where the equivalence means that the two operators have the same eigenfunctions and eigenvalues. 
 
Similar results exist in the theory of Dunkl operators. These operators (see \cite{MR951883, MR1827871}) are a set of 
differential-difference operators of first order that generate a commutative algebra. The Dunkl Laplacian $\Delta_{k}$, playing the 
role similar to that of the ordinary Laplacian, is a second order operator in the center of the algebra, invariant under a finite reflection group $\cG < O(m)$, instead of the entire rotation group $O(m)$. It can be proven that $\Delta_{k}$ together with $|\ux|^{2}$ again
generates $\mathfrak{sl}_{2}$ (see \cite{He}). The Dunkl transform was introduced in \cite{Du4} and further studied in \cite{deJ}. 
Although for general reflection groups, the kernel of this integral transform is not explicitly known, it can equivalently be expressed by a similar operator exponential expression as follows (see \cite{Said}, formula (3.19))
\[
\cF_{k}= e^{ \frac{i \pi \mu}{4}} e^{\frac{i \pi}{4}(\Delta_{k} - |\ux|^{2})} 
\]
with $\mu$ a constant related to the reflection group under consideration.

More recently, further generalizations of the classical Fourier transform and the Dunkl transform have been introduced. This was first done in the context of minimal representations (see e.g. \cite{KM1, KM2, KM3}) and subsequently generalized to so-called radial deformations (see e.g. \cite{Orsted2, H12}). Again these Fourier transforms are given by similar operator exponentials containing generators of $\mathfrak{sl}_{2}$. It should be remarked that it is in general a difficult question to obtain an explicit integral kernel for such operator exponentials. In \cite{KM3} the kernel is determined very explicitly for the case of a Fourier transform on an isotropic cone (associated to an indefinite quadratic form of signature $(p,q)$), but in \cite{Orsted2, H12} explicit expressions are only obtained for a few values of the deformation parameters and one is restricted to using series expansions for general parameter values.

Clifford analysis (see e.g. \cite{MR697564, MR1169463}) is a refinement of harmonic analysis in $\mR^{m}$, in the sense that the $\mathfrak{sl}_{2}$ algebra generated by the basic operators in harmonic analysis is refined to the Lie superalgebra $\mathfrak{osp}(1|2)$ (containing $\mathfrak{sl}_{2}$ as its even subalgebra). This is obtained by introducing the Dirac operator
\begin{equation}
\upx :=  \sum_{i=1}^{m} e_{i}\pI
\label{DiracOperator}
\end{equation}
and the vector variable
\begin{equation}
\ux := \sum_{i=1}^{m} e_{i} x_{i},
\label{VV}
\end{equation}
where $e_1, \ldots, e_m$ generate the Clifford algebra $\cC l_{0,m}$ with the relations $e_{i} e_{j} + e_{i} e_{j} = -2 \delta_{ij}$. These operators satisfy $\upx^{2} = - \Delta$ and $\ux^{2} = - |\ux|^{2}$. The other important relations will be given in Theorem \ref{ospFamily} below.

Several attempts have been made to introduce a generalization of the Fourier transform to the setting of Clifford analysis (see \cite{AIEP} for a review). From our point of view, only the so-called Clifford-Fourier transform introduced in \cite{MR2190678} is promising, because it is given by a similar operator exponential as the classical Fourier transform in \eqref{expclassFourier}, but now containing generators of $\mathfrak{osp}(1|2)$. More precisely, it is defined by
\begin{equation}\label{clifford-fourier}
\cF_{\pm} =  e^{ \frac{i \pi m}{4}} e^{\mp \frac{i \pi}{2}\Gamma  }e^{\frac{i \pi}{4}(\Delta - |\ux|^{2})} = e^{ \frac{i \pi m}{4}} e^{\frac{i \pi}{4}(\Delta - |\ux|^{2} \mp 2\Gamma)}
\end{equation}
with $2\Gamma= (\upx \ux - \ux \upx)+m$. Note that we use a slightly different normalization as 
in \cite{MR2190678}. The main problem concerning this transform is to write it as an integral transform 
\[
\cF_{\pm} (f)(y) = (2 \pi)^{-\frac{m}{2}} \int_{\mR^{m}} K_\pm(x,y) f(x) \, dx,
\]
where the kernel function $K_{\pm}(x,y)$ is given by 
$$
    K_\pm (x,y) = e^{\mp i \frac{\pi}{2} \Gamma_{\uy} } e^{-i \la \ux,\uy\ra}
$$
and to find an explicit expression for $K_\pm(x,y)$. So far the kernel is found explicitly only in the case $m=2$ (see \cite{MR2283868}). For higher even dimensions, a complicated iterative procedure for constructing the kernel is given in \cite{JMAAFourierBessel}, which can be used to explicitly compute the kernel only for low dimensions (say $m=4, 6$). At the moment, no results are known in the case of odd $m$. 

Our main result in this paper is the following. 
We will give a completely explicit description of the kernel in terms of a finite sum of Bessel functions when $m$ is even (see Theorem \ref{EvenExplicit}). In the case of odd $m$, we are able to show that it is enough to identify the kernel in dimension three, from which kernels in higher odd dimensions can be deduced by taking suitable derivatives (see Theorem \ref{KernelOdd}), and we are able to express the kernel in dimension three as a single integral of a combination of Bessel functions. 

The compact formula of the kernel shows that $K(x,y)$ is unbounded for even dimensions $m > 2$, in sharp contrast to the classical Fourier transform, and it yields a sharp bound for the kernel. This allows us to establish the inversion formula for the 
Clifford-Fourier transform. In the process, we define a generalized translation operator in terms of the Clifford-Fourier transform and 
a generalized convolution in terms of the translation operator. It turns out, rather surprisingly, that the generalized translation 
coincides with the ordinary translation when the function being translated is radial. These results indicate a possible theory of harmonic analysis for Clifford algebra akin to the classical harmonic analysis. 

The paper is organized as follows. In section \ref{prelim} we recall basic facts on Clifford algebras and Dirac operators necessary for the sequel. In section \ref{CFdev} we derive a series representation for kernel of the Clifford-Fourier transform that is valid for all dimensions. In section \ref{CFeven} we obtain the explicit expression for the kernel in even  dimensions. In section \ref{CFodd} we discuss the kernel in odd dimensions. In section \ref{secFurtherprops} we obtain some properties of the kernels and obtain the necessary bounds. In section \ref{Inversion}, we show that the Clifford-Fourier transform is a continuous operator on Schwartz class functions and show that it coincides with the exponential operator introduced in \cite{MR2190678}. Next, in section \ref{SecTranslation}, we define a translation operator related to the Clifford-Fourier transform and obtain the important result that this translation operator coincides with the classical translation for radial functions. Finally, in section \ref{SecConvolution} we introduce a convolution structure based on the translation operator, which allows us to prove the inversion theorem for a broader class of functions.

\section{Preliminaries}\label{prelim}
\setcounter{equation}{0}

The Clifford algebra $\cC l_{0,m}$ over $\mR^{m}$ is the algebra generated by $e_{i}$, $i= 1, \ldots, m$, under the relations
\begin{align} \label{eq:eij}
\begin{split}
&e_{i} e_{j} + e_{i} e_{j} = 0, \qquad i \neq j,\\
& e_{i}^{2} = -1.
\end{split}
\end{align}
This algebra has dimension $2^{m}$ as a vector space over $\mR$. It can be decomposed as $\cC l_{0,m} = \oplus_{k=0}^{m} \cC l_{0,m}^{k}$
with $\cC l_{0,m}^{k}$ the space of $k$-vectors defined by
\[
\cC l_{0,m}^{k} := \mbox{span} \{ e_{i_{1}} \ldots e_{i_{k}}, i_{1} < \ldots < i_{k} \}.
\]
In the sequel, we will always consider functions $f$ taking values in $\cC l_{0,m}$, unless explicitly mentioned. Such functions can be decomposed as
\begin{equation} \label{clifford_func}
f(x) = f_{0}(x) + \sum_{i=1}^{m} e_{i}f_{i}(x) + \sum_{i< j} e_{i} e_{j} f_{ij}(x) + \ldots + e_{1} \ldots e_{m} f_{1 \ldots m} (x)
\end{equation}
with $f_{0}, f_{i}, f_{ij}, \ldots, f_{1 \ldots m}$ all real-valued functions.

The Dirac operator (\ref{DiracOperator}) and the vector variable (\ref{VV}) together generate the Lie superalgebra $\mathfrak{osp}(1|2)$. This is the subject of the following theorem.

\begin{thm}
The operators $\upx$ and $\ux$ generate a Lie superalgebra, isomorphic with $\mathfrak{osp}(1|2)$, with the following relations
\begin{equation}
\begin{array}{lll}
\{ \ux, \ux \} = -2 |\ux|^{2} & \quad & \{ \upx, \upx \} = -2 \Delta\\
\vspace{-3mm}\\
\{ \ux, \upx\} = -2 \left( \mE + \frac{m}{2}\right)&\quad&\left[\mE + \frac{m}{2}, \upx \right] = - \upx\\
\vspace{-3mm}\\
\left[ |\ux|^{2}, \upx \right] = -2\ux&\quad&\left[\mE + \frac{m}{2}, \ux \right] =   \ux\\
\vspace{-3mm}\\
\left[ \Delta, \ux \right] = 2\upx&\quad&\left[\mE + \frac{m}{2}, \Delta \right] = - 2 \Delta\\
\vspace{-3mm}\\
\left[ \Delta^{2}, |\ux|^{2} \right] = 4 \left( \mE + \frac{m}{2} \right)&\quad&\left[\mE + \frac{m}{2}, |\ux|^{2} \right] = 2 |\ux|^{2},
\end{array}
\end{equation}
where $\mE = \sum_{i=1}^{m}x_{i}\partial_{x_{i}}$ is the Euler operator.
\label{ospFamily}
\end{thm}

\begin{proof}
This is the special case $a=2$ and $b=c=k=0$ of Theorem 1 in \cite{H12}. The same result was proven earlier in e.g. Proposition 3.1  in \cite{HS}.
\end{proof}

The classical Laplace operator $\Delta$ and $|\ux|^{2}$ together generate $\mathfrak{sl}_{2}$ and this Lie algebra is the even subalgebra of $\mathfrak{osp}(1|2)$ (\cite{H12}). 

We further introduce the so-called Gamma operator (see e.g. \cite{MR1169463})
\[
\Gamma_{\ux} := - \sum_{j<k} e_{j}e_{k} (x_{j} \partial_{x_{k}} - x_{k}\partial_{x_{j}}) = -\ux \upx - \mE. 
\]
The operators $x_{j} \partial_{x_{k}} - x_{k}\partial_{x_{j}}$ are also called Euler angles (see \cite{V}).
Note that $\Gamma_{\ux}$ commutes with radial functions, i.e. $[\Gamma_{\ux}, f(|\ux|)] =0$.

Denote by $\cP$ the space of polynomials taking values in $\cC l_{0,m}$, i.e. 
$$
   \cP : = \mR[x_{1}, \ldots, x_{m}] \otimes \cC l_{0,m}.
$$ 
The space of homogeneous polynomials of degree $k$ is then denoted by $\cP_{k}$. 
The space $\cM_{k} : = \ker{\upx} \cap \cP_{k}$ is called the space of spherical monogenics 
of degree $k$. Similarly,  $\cH_{k} := \ker{\Delta} \cap \cP_{k}$ is the space of spherical 
harmonics of degree $k$. 

The elements of $\cH_k$ are functions of the form \eqref{clifford_func}
with $f_{i_1,\ldots,i_j}$ being ordinary harmonics. It follows, in particular, that the reproducing kernel
of $\cH_k$ is $\frac{\lambda+k}{\l} C^{\lambda}_{k}(\langle \xi,\eta \rangle)$, with $\l= (m-2)/2$ and 
$C_k^\lambda$ being the Gegenbauer polynomial, the same reproducing kernel for the space of 
ordinary spherical harmonics of degree $k$ (cf. \cite{MR1827871,V}). This means that
\begin{equation} \label{FH}
\frac{\lambda+k}{\l} \int_{\mS^{m-1}} C^{\lambda}_{k}(\langle \xi,\eta \rangle) H_{\ell}(\xi) d \sigma(\xi) = c \, \delta_{k \ell} \, H_{\ell}(\eta), \quad H_{\ell} \in \cH_{\ell}
\end{equation}
with $c = \frac{2 \pi^{\frac{m}{2}}}{\Gamma(m/2)}$. The definition shows immediately that $\cM_{k} \subset \cH_{k}$. More precisely, we have the following Fischer decomposition (see \cite[Theorem 1.10.1]{MR1169463}):
\begin{equation}
\cH_{k} = \cM_{k} \oplus \ux \cM_{k-1}.
\label{Fischer}
\end{equation}
It is easy to construct projection operators to the components of this decomposition. They are given by 
(\cite[Corollary 1.3.3]{MR1169463})
\begin{align*}
\mP_{1} &= 1+\frac{\ux \upx}{2k+m-2}, \\
\mP_{2} &= -\frac{\ux \upx}{2k+m-2}.   
\end{align*}
and they satisfy $\mP_{1}+ \mP_{2} = 1$, $\mP_{1}\cH_{k} = \cM_{k}$ and $\mP_{2}\cH_{k} = \ux \cM_{k-1}$.
We also have the relations 
\begin{align} \label{gammaeig1}
\Gamma_{\ux} \cM_{k} &= -k \cM_{k},\\
\label{gammaeig2}
\Gamma_{\ux} ( \ux \cM_{k-1})&= (k+m -2) \ux \cM_{k-1},
\end{align}
which follows easily from $\Gamma_{\ux} = -\ux \upx - \mE$ and Theorem \ref{ospFamily}.

In the sequel we will often use the following well-known properties of Gegenbauer polynomials (see \cite{Sz}):
\begin{equation}
\frac{\l + n}{\l} C^{\l}_{n}(w) =  C_{n}^{\l+1}(w) - C_{n-2}^{\l+1}(w)
\label{Geg3}
\end{equation}
and
\begin{equation}
w C^{\l+1}_{n-1}(w) = \frac{n}{2(n+ \l)} C^{\l+1}_{n}(w) + \frac{n + 2 \l}{2(n+ \l)} C^{\l+1}_{n-2}(w).
\label{Geg2}
\end{equation}
From (\ref{Geg3}) we immediately obtain
\begin{equation}
C^{\l+1}_{n}(w) = \sum_{k=0}^{\lfloor \frac{n}{2}\rfloor} \frac{\l + n-2k}{\l} C^{\l}_{n-2k}(w).
\label{Geg1}
\end{equation}

Now we define the inner product and the wedge product of two vectors $\ux$ and $\uy$
\begin{align*}
\langle \ux, \uy \rangle &:= \sum_{j=1}^{m} x_{j} y_{j} = - \frac{1}{2} (\ux \uy + \uy  \ux)\\
\ux \wedge \uy &:= \sum_{j<k} e_{j}e_{k} (x_{j} y_{k} - x_{k}y_{j}) = \frac{1}{2} (\ux \uy - \uy  \ux).
\end{align*}
Here, the multiplication of vectors $\ux$ and $\uy$ of the form \eqref{VV} is defined using the relations  \eqref{eq:eij} of the Clifford algebra.

For the sequel we need the square of $\ux \wedge \uy$. This is most easily calculated as follows
\begin{align*}
(\ux \wedge \uy )^{2} &= \frac{1}{4} (\ux \uy - \uy  \ux) (\ux \uy - \uy  \ux)\\
&= \frac{1}{4}(\ux \uy \ux \uy +\uy \ux  \uy  \ux -2 |\ux|^{2}|\uy|^{2})\\
&= \frac{1}{4} (-2 \langle \ux,\uy\rangle \ux\uy -2 \langle \ux,\uy\rangle \uy\ux -4 |\ux|^{2}|\uy|^{2})\\
&= - |\ux|^{2}|\uy|^{2} +\langle \ux,\uy\rangle^{2},
\end{align*}
from which we see that $(\ux \wedge \uy )^{2}$ is real-valued. As a consequence, we also obtain that
\[
(\ux \wedge \uy )^{2} = - \sum_{j<k}  (x_{j} y_{k} - x_{k}y_{j})^{2}.
\]
In turn, this allows us to estimate
\begin{equation}
\left| \frac{x_{j} y_{k} - x_{k}y_{j}}{\sqrt{|\ux|^{2}|\uy|^{2} -\langle \ux,\uy\rangle^{2}}} \right| \leq 1, \qquad \forall \ux, \uy \in \mR^{m}.
\label{estimateCterm}
\end{equation}

Let us fix some notations for the Clifford-Fourier transform and its inverse that we wish to study. The operator exponential definition of the Clifford-Fourier transform is given by (see \cite{MR2190678}, we follow the normalization given in \cite{AIEP})
\begin{equation}
e^{ \frac{i \pi m}{4}} e^{\frac{i \pi}{4}(\Delta - |\ux|^{2} \mp 2\Gamma)} = e^{ \frac{i \pi m}{4}} e^{\mp \frac{i \pi}{2}\Gamma  }e^{\frac{i \pi}{4}(\Delta - |\ux|^{2})}.
\label{CFTF}
\end{equation}
The second equality follows because $\Gamma$ commutes with $\Delta$ and $|\ux|^{2}$.
Formally, we obtain its inverse as follows
\[
e^{-  \frac{i \pi m}{4}} e^{-\frac{i \pi}{4}(\Delta - |\ux|^{2} \mp 2\Gamma )}.
\]

We introduce a basis $\{ \psi_{j,k,l}\}$ for the space $\cS(\mR^{m}) \otimes \cC l_{0,m}$, where 
$\cS(\mR^m)$ denotes the Schwartz space.  This basis is defined by 
\begin{align} \label{basis}
\begin{split}
\psi_{2j,k,l}(x) &:= L_{j}^{\frac{m}{2}+k-1}(|\ux|^{2}) M_{k}^{(l)} e^{-|\ux|^{2}/2},\\
\psi_{2j+1,k,l}(x) &:= L_{j}^{\frac{m}{2}+k}(|\ux|^{2}) \ux M_{k}^{(l)} e^{-|\ux|^{2}/2},
\end{split}
\end{align}
where $j,k \in \mN$, $\{M_{k}^{(l)} \in \cM_{k}: l= 1, \ldots, \dim \cM_{k}\}$ is a basis for $\cM_{k}$,
and $L_{j}^{\alpha}$ are the Laguerre polynomials. The set $\{ \psi_{j,k,l}\}$ forms a basis of 
$\cS(\mR^{m}) \otimes \cC l_{0,m}$ (as can be seen from (\ref{Fischer})). The action of the 
Clifford-Fourier transform on this basis is given by (see \cite{MR2190678}) 
\begin{align}
\label{EigCF1}
\begin{split}
e^{ \frac{i \pi m}{4}} e^{\mp \frac{i \pi}{2}\Gamma  }e^{\frac{i \pi}{4}(\Delta - |\ux|^{2})} (\psi_{2j,k,l}) &= (-1)^{j+k} (\mp 1)^{k} \psi_{2j,k,l},\\
e^{ \frac{i \pi m}{4}} e^{\mp \frac{i \pi}{2}\Gamma  }e^{\frac{i \pi}{4}(\Delta - |\ux|^{2})} (\psi_{2j+1,k,l}) &= i^{m} (-1)^{j+1} (\mp 1 )^{k+m-1} \psi_{2j+1,k,l}.
\end{split}
\end{align}
Observe that if the dimension $m$ is even, there are two eigenvalues $\pm 1$. If the dimension is odd, there are four eigenvalues, namely $\pm 1$ and $\pm i$.  

Combining formulas (\ref{classFourier}) and (\ref{expclassFourier}) for the classical Fourier transform with the definition of the Clifford-Fourier transform (\ref{CFTF}) we obtain
\[
e^{ \frac{i \pi m}{4}} e^{\mp \frac{i \pi}{2}\Gamma  }e^{\frac{i \pi}{4}(\Delta - |\ux|^{2})} \sim (2 \pi)^{-\frac{m}{2}}  e^{\mp i \frac{\pi}{2}\Gamma_{\uy}} \int_{\mR^{m}} e^{-i \la \ux, \uy \ra} f(x) dx.
\]
Heuristically, this suggests that the Clifford-Fourier transform can be defined as an integral transform with 
$e^{i \frac{\pi}{2}\Gamma_{\uy}} e^{-i \langle \ux,\uy \rangle}$ as its kernel. We give a formal definition as follows. 

\begin{defn} \label{def:FC}
On the Schwartz class of functions $\cS(\RR^m) \otimes \cC l_{0,m}$, we define
\begin{align*}
\cF_{\pm}f(y) & : = (2 \pi)^{-\frac{m}{2}} \int_{\mR^{m}} K_\pm(x,y) f(x) \, dx,\\
\cF_{\pm}^{-1} f(y) & :=(2 \pi)^{-\frac{m}{2}} \int_{\mR^{m}} \widetilde{K_\pm}(x,y) f(x) \, dx,
\end{align*}
where
\begin{align*}
K_\pm(x,y)& := e^{\mp i \frac{\pi}{2}\Gamma_{\uy}} e^{-i \langle \ux,\uy \rangle},\\
\widetilde{K_\pm}(x,y)& := e^{ \pm i \frac{\pi}{2}\Gamma_{\uy}} e^{i \langle \ux,\uy \rangle}.
\end{align*}
\end{defn}

Since the kernels  $K_\pm$ are defined via an exponential differential operator, it is not immediately clear if the kernel is bounded. As a result, it is not clear if $\cF_\pm$ is well-defined even on $\cS(\RR^m)$ and neither is it clear if $\cF_\pm^{-1}$ is indeed the inversion of $\cF_\pm$. The main purpose of the paper is to answer these questions.

\begin{rem}
It is important to note that the kernels are not symmetric, in the sense that $K(x,y) \neq K(y,x)$ (see e.g. Theorem \ref{seriesthm}
below). Hence, we adopt the convention that we always integrate over the first variable in the kernel.
\end{rem}

\section{Series representation of the kernel}
\label{CFdev}
\setcounter{equation}{0}

In this section we derive series representations of the kernels introduced in the previous section. We first focus on $K_-(x,y)=e^{i \frac{\pi}{2}\Gamma_{\uy}} e^{-i \langle \ux,\uy \rangle}$.

We start from the decomposition of the classical Fourier kernel in $\mR^{m}$ in terms of Gegenbauer polynomials and Bessel functions (see \cite[Section 11.5]{MR0010746})
\[
e^{-i \langle \ux,\uy \rangle} = 2^{\lambda} \Gamma(\lambda)\sum_{k=0}^{\infty}(k+ \lambda) (-i)^{k} (|\ux||\uy|)^{-k-\lambda} J_{k+ \lambda}(|\ux||\uy|) \; (|\ux||\uy|)^{k} C_{k}^{\lambda}(\langle \xi,\eta \rangle),
\]
where $\xi= \ux/|\ux|$, $\eta = \uy/|\uy|$ and $\lambda =(m-2)/2$. Then, using the fact that $\Gamma_{\uy}$ commutes with radial functions, we have 
\begin{align}\label{GammaWatson}
 e^{i \frac{\pi}{2}\Gamma_{\uy}} e^{-i \langle \ux,\uy \rangle} =& \, 2^{\lambda} \Gamma(\lambda)\sum_{k=0}^{\infty}(k+ \lambda) (-i)^{k} (|\ux||\uy|)^{-k-\lambda} J_{k+ \lambda}(|\ux||\uy|) \\
&\times e^{i \frac{\pi}{2}\Gamma_{\uy}}\left[ (|\ux||\uy|)^{k} C_{k}^{\lambda}(\langle \xi,\eta \rangle)\right].
\notag
\end{align}

We now prove the following lemma.
\begin{lem}\label{actionGamma}
For $\ux = |\ux| \xi$ and $\uy = |\uy| \eta$, 
\begin{align*}
e^{i \frac{\pi}{2}\Gamma_{\uy}}(|\ux||\uy|)^{k} C_{k}^{\lambda}(\langle \xi,\eta \rangle) = \, & \frac{1}{2}(i^{k+m-2} + i^{-k}) (|\ux||\uy|)^{k} C_{k}^{\lambda}(\langle \xi,\eta \rangle)\\
&  - \frac{\lambda }{2(k+\lambda)} (i^{k+m-2} - i^{-k}) (|\ux||\uy|)^{k} C_{k}^{\lambda}(\langle \xi,\eta \rangle)\\
& +  \frac{ \lambda  }{k + \lambda} \ux \wedge \uy (  i^{k+m-2} -i^{-k}) (|\ux||\uy|)^{k-1} C_{k-1}^{\lambda+1}(\langle \xi,\eta \rangle).
\end{align*}
\end{lem}

\begin{proof}
Note that $(|\ux||\uy|)^{k} C_{k}^{\lambda}(\langle \xi,\eta \rangle)$ is a harmonic homogeneous polynomial of degree $k$. It can hence be decomposed into monogenic components (with respect to the variables $y$) according to (\ref{Fischer}).  In order to obtain this decomposition, we first calculate
\begin{align*}
\upy \left( (|\ux||\uy|)^{k} C_{k}^{\lambda}(\langle \xi,\eta \rangle) \right) =& \, k (|\ux||\uy|)^{k} \frac{\uy}{|\uy|^{2}}C_{k}^{\lambda}(\langle \xi,\eta \rangle)\\
& +  (|\ux||\uy|)^{k} \left(\frac{\ux}{|\ux||\uy|} -\frac{\langle \ux,\uy \rangle}{|\ux||\uy|} \frac{\uy}{|\uy|^{2}} \right)\left(\frac{d}{dt}C_{k}^{\lambda}\right)(\langle \xi,\eta \rangle) \\
=& \, k (|\ux||\uy|)^{k} \frac{\uy}{|\uy|^{2}}C_{k}^{\lambda}(\langle \xi,\eta \rangle)\\
& + 2 \lambda (|\ux||\uy|)^{k-1} \left(\ux -\langle \ux,\uy \rangle \frac{\uy}{|\uy|^{2}} \right) C_{k-1}^{\lambda+1}(\langle \xi,\eta \rangle), 
\end{align*}
where we have used $\frac{d}{d t} C^{\lambda}_{k}(t) = 2 \lambda  C^{\lambda+1}_{k-1}(t)$. This yields
\begin{align*}
\uy \upy \left( (|\ux||\uy|)^{k} C_{k}^{\lambda}(\langle \xi,\eta \rangle) \right) =& \, - k (|\ux||\uy|)^{k} C_{k}^{\lambda}(\langle \xi,\eta \rangle)\\
& +  2 \lambda (|\ux||\uy|)^{k-1} \left(\uy \ux +\langle \ux,\uy \rangle  \right) C_{k-1}^{\lambda+1}(\langle \xi,\eta \rangle) \\
=& - k (|\ux||\uy|)^{k} C_{k}^{\lambda}(\langle \xi,\eta \rangle)\\
&  - 2 \lambda (|\ux||\uy|)^{k-1}\left(\ux \wedge \uy \right) C_{k-1}^{\lambda+1}(\langle \xi,\eta \rangle). 
\end{align*}
Thus, we have that $(|\ux||\uy|)^{k} C_{k}^{\lambda}(\langle \xi,\eta \rangle) = F_{k} + G_{k}$ with $F_{k} \in \cM_{k}$ and $G_{k} \in \uy \cM_{k-1}$ given by
\begin{align*}
F_{k} =& \, \left(1-\frac{k}{2k+m-2}   \right) (|\ux||\uy|)^{k} C_{k}^{\lambda}(\langle \xi,\eta \rangle)\\
& \quad  - \frac{2 \lambda}{2k+m-2}(|\ux||\uy|)^{k-1}\left(\ux \wedge \uy \right) C_{k-1}^{\lambda+1}(\langle \xi,\eta \rangle), \\
G_{k} =& \, \frac{1}{2k+m-2} \left[  k (|\ux||\uy|)^{k} C_{k}^{\lambda}(\langle \xi,\eta \rangle)  + 2 \lambda(|\ux||\uy|)^{k-1}\left(\ux \wedge \uy \right) C_{k-1}^{\lambda+1}(\langle \xi,\eta \rangle) \right].
\end{align*}

Using (\ref{gammaeig1}) and (\ref{gammaeig2}) we obtain 
\begin{align*}
e^{i \frac{\pi}{2}\Gamma_{\uy}} F_{k} = &\, i^{-k} F_{k}, \\
e^{i \frac{\pi}{2}\Gamma_{\uy}} G_{k} =& \, i^{k+m-2} G_{k}.
\end{align*}
We can hence calculate
\begin{align*}
 e^{i \frac{\pi}{2}  \Gamma_{\uy}} & (|\ux||\uy|)^{k} C_{k}^{\lambda}(\langle \xi,\eta \rangle)\\
 = &\, i^{-k} \left(1-\frac{k}{2k+m-2}   \right) (|\ux||\uy|)^{k} C_{k}^{\lambda}(\langle \xi,\eta \rangle)\\
 &  -  i^{-k} \frac{2 \lambda}{2k+m-2}(|\ux||\uy|)^{k-1}\left(\ux \wedge \uy \right) C_{k-1}^{\lambda+1}(\langle \xi,\eta \rangle)\\
&+ \frac{i^{k+m-2}}{2k+m-2} \left[  k (|\ux||\uy|)^{k} C_{k}^{\lambda}(\langle \xi,\eta \rangle)  + 2 \lambda(|\ux||\uy|)^{k-1}\left(\ux \wedge \uy \right) C_{k-1}^{\lambda+1}(\langle \xi,\eta \rangle) \right]\\
= &\, i^{-k} (|\ux||\uy|)^{k} C_{k}^{\lambda}(\langle \xi,\eta \rangle)  + \frac{k }{2k+m-2} (i^{k+m-2} - i^{-k}) (|\ux||\uy|)^{k} C_{k}^{\lambda}(\langle \xi,\eta \rangle)\\
& +  \frac{2 \lambda}{2k+m-2}(  i^{k+m-2} -i^{-k}) (|\ux||\uy|)^{k-1}\left(\ux \wedge \uy \right) C_{k-1}^{\lambda+1}(\langle \xi,\eta \rangle)\\
=&\,  i^{-k} (|\ux||\uy|)^{k} C_{k}^{\lambda}(\langle \xi,\eta \rangle)  + \frac{k }{2(k+\lambda)} (i^{k+m-2} - i^{-k}) (|\ux||\uy|)^{k} C_{k}^{\lambda}(\langle \xi,\eta \rangle)\\
&+  \frac{2 \lambda}{2(k + \lambda)}(  i^{k+m-2} -i^{-k}) (|\ux||\uy|)^{k-1}\left(\ux \wedge \uy \right) C_{k-1}^{\lambda+1}(\langle \xi,\eta \rangle)\\
=&\,  \frac{1}{2}(i^{k+m-2} + i^{-k}) (|\ux||\uy|)^{k} C_{k}^{\lambda}(\langle \xi,\eta \rangle)\\
&  - \frac{\lambda }{2(k+\lambda)} (i^{k+m-2} - i^{-k}) (|\ux||\uy|)^{k} C_{k}^{\lambda}(\langle \xi,\eta \rangle)\\
& +  \frac{ \lambda}{k + \lambda}(  i^{k+m-2} -i^{-k}) (|\ux||\uy|)^{k-1}\left(\ux \wedge \uy \right) C_{k-1}^{\lambda+1}(\langle \xi,\eta \rangle),
\end{align*}
thus completing the proof of the lemma.
\end{proof}

Combining formula (\ref{GammaWatson}) with Lemma \ref{actionGamma}, we find
\[
K_-(x,y) = e^{i \frac{\pi}{2}\Gamma_{\uy}} e^{-i \langle \ux,\uy \rangle} = A_{\lambda} + B_{\lambda} + \left(\ux \wedge \uy \right) C_{\lambda}
\]
with
\begin{align*}
A_{\lambda} :=& \, 2^{\lambda-1} \Gamma(\lambda+1)\sum_{k=0}^{\infty}  (i^{m} + (-1)^{k})  (|\ux||\uy|)^{-\lambda} J_{k+ \lambda}(|\ux||\uy|) \;  C_{k}^{\lambda}(\langle \xi,\eta \rangle),\\
B_{\lambda} :=& \, - 2^{\lambda-1} \Gamma(\lambda)\sum_{k=0}^{\infty}(k+ \lambda)  (i^{m} - (-1)^{k}) (|\ux||\uy|)^{-\lambda} J_{k+ \lambda}(|\ux||\uy|) \; C_{k}^{\lambda}(\langle \xi,\eta \rangle),\\
C_{\lambda} :=& \, - (2 \lambda) 2^{\lambda-1} \Gamma(\lambda)\sum_{k=1}^{\infty} (i^{m} + (-1)^{k}) (|\ux||\uy|)^{-\lambda-1} J_{k+ \lambda}(|\ux||\uy|) \;  C_{k-1}^{\lambda+1} (\langle \xi,\eta \rangle).
\end{align*}

Introducing new variables $z = |\ux||\uy|$ and $w = \langle \xi,\eta \rangle$ to simplify notations and substituting $m = 2 \lambda + 2$, we have thus obtained the following theorem.

\begin{thm}
The kernel of the Clifford-Fourier transform is given by
\[
K_-(x,y) =  A_{\lambda} + B_{\lambda} + \left(\ux \wedge \uy \right) C_{\lambda}
\]
with
\begin{align*}
A_{\lambda}(w,z) =& \,2^{\lambda-1} \Gamma(\lambda+1)\sum_{k=0}^{\infty}  (i^{2 \lambda + 2} + (-1)^{k})  z^{-\lambda} J_{k+ \lambda}(z) \;  C_{k}^{\lambda}(w),\\
B_{\lambda}(w,z) = & \, - 2^{\lambda-1} \Gamma(\lambda)\sum_{k=0}^{\infty}(k+ \lambda)  (i^{2 \lambda + 2} - (-1)^{k}) z^{-\lambda} J_{k+ \lambda}(z) \; C_{k}^{\lambda}(w),\\
C_{\lambda}(w,z) = & \, -  2^{\lambda-1} \Gamma(\lambda)\sum_{k=0}^{\infty} (i^{2 \lambda + 2} + (-1)^{k}) z^{-\lambda-1} J_{k+ \lambda}(z) \;  \left(\frac{d}{dw}C_{k}^{\lambda}\right) (w),
\end{align*}
where $z = |\ux||\uy|$ and $w = \langle \xi,\eta \rangle$.
\label{seriesthm}
\end{thm}

The functions $A_{\lambda}$, $B_{\lambda}$ and $C_{\lambda}$ satisfy nice recursive relations. They are given in the following lemma.
\begin{lem} \label{RecursionsProps}
For $m \ge 2$, or equivalently, $\lambda \ge 1$, one has
\begin{align*}
A_{\lambda}(w,z) =& \, - \frac{\lambda}{\lambda -1} \frac{1}{z} \partial_{w}A_{\lambda-1}(w,z),\\
B_{\lambda}(w,z) =& \, -  \frac{1}{z} \partial_{w}B_{\lambda-1}(w,z),\\
C_{\lambda}(w,z) =& \, -\frac{1}{\lambda z} \partial_{w} A_{\lambda}(w,z).
\end{align*}
\end{lem}

\begin{proof}
We start with the relation for $A_{\lambda}(w,z)$. We rewrite this function as $A_{\lambda}(w,z) = A_{\lambda}^{ \mbox{\small odd}}(w,z) + A_{\lambda}^{\mbox{\small even}}(w,z)$ with
\begin{align*}
A_{\lambda}^{ \mbox{\small odd}}(w,z) &=\, -2^{\lambda-1} \Gamma(\lambda+1)  (i^{2 \lambda} +1)  z^{-\lambda}\sum_{k=0}^{\infty} J_{2k+ \lambda+1}(z) \;  C_{2k+1}^{\lambda}(w),\\
A_{\lambda}^{\mbox{ \small even}}(w,z) &=\, 2^{\lambda-1} \Gamma(\lambda+1)  (1- i^{2 \lambda })  z^{-\lambda}\sum_{k=0}^{\infty} J_{2k+ \lambda}(z) \;  C_{2k}^{\lambda}(w).
\end{align*}
We then calculate, using $\frac{d}{d w} C^{\lambda}_{k}(w) = 2 \lambda  C^{\lambda+1}_{k-1}(w)$, that
\begin{align*}
 \partial_{w}A_{\lambda}^{\mbox{ \small odd}}(w,z) =&\, -2^{\lambda} \lambda \Gamma(\lambda+1)  (i^{2 \lambda} +1)  z^{-\lambda}\sum_{k=0}^{\infty} J_{2k+ \lambda+1}(z) \;  C_{2k}^{\lambda+1}(w)\\
=&\, 2^{\lambda} \frac{\lambda}{\lambda+1} \Gamma(\lambda+2)  (i^{2 \lambda+2} -1)  z^{-\lambda}\sum_{k=0}^{\infty} J_{2k+ \lambda+1}(z) \;  C_{2k}^{\lambda+1}(w)\\
=&\, - \frac{\lambda}{\lambda+1} z A_{\lambda+1}^{\mbox{ \small even}}(w,z)
\end{align*}
and similarly $\partial_{w}A_{\lambda}^{\mbox{ \small even}}(w,z) =  -\frac{\lambda}{\lambda+1} z A_{\lambda+1}^{\mbox{ \small odd}}(w,z)$. This completes the proof of the first statement.

The proof for $B_{\lambda}$ is similar and the third statement is trivial.
\end{proof}

Using exactly the same line of reasoning as leading to Theorem \ref{seriesthm}, we can also calculate the kernel $K_+(x,y)$ and the inverse kernels $\widetilde{K_\pm}(x,y)$. It follows that these kernels satisfy the following relations: 

\begin{prop}
\label{kernelsrels}
For $x, y \in \RR^m$, 
\begin{align*}
K_+(x,y) &= \overline{K_-(x,-y)},\\
\widetilde{K_\pm}(x,y) &= \overline{K_\pm(x,y)}.
\end{align*}
In particular, in the case $m$ even, $K_-(x,y)$ is real-valued and the complex conjugation can be omitted.
\end{prop}

As a consequence, it suffices to work with $K_-(x,y)$, which will be determined explicitly in the following section. 

\section{Explicit representation of the kernel}
\setcounter{equation}{0}
\label{ExplicitCF}

We determine the explicit formula of $K_-(x,y)$ on $\RR^m$. It turns out that there is a distinct difference between $m$ being even and $m$ being odd. 
\subsection{The case $m$ even}
\label{CFeven}

\subsubsection{The case $m=2$}

In this case $\lambda =0$. We need the well-known relation \cite[(4.7.8)]{Sz} 
$$
\lim_{\lambda \rightarrow 0}  \lambda^{-1} C_n^\lambda (w) = (2/n) \cos n \theta, \quad w = \cos \theta, \quad n \geq 1.   
$$
Then it is easy to deduce that $A_0(w,z) =0$ and
\begin{align*}
 B_0(w,z) & = J_0(z) + 2 \sum_{k=1}^\infty J_{2k}(z) \cos (2k \theta), \\
 C_0(w,z) & = \frac{2}{z} \sum_{k=1}^\infty J_{2k-1}(z)  \frac{\sin (2k -1)\theta}{\sin \theta}, 
\end{align*}
where in the last series we have used the fact that $\frac{d}{d w} \cos n \theta = n \sin n \theta/ \sin \theta$.
By \cite[p. 7, formula (26)]{Er}, we have
$$
   e^{i z \sin \theta}= J_0(z) + 2 \sum_{k=1}^\infty J_{2k}(z) \cos (2k \theta) + 
     2 i \sum_{k=1}^\infty J_{2k-1}(z)  \frac{\sin (2k -1)\theta}{\sin \theta}.
$$
Taking real and imaginary parts gives 
$$     
     B_0(w,z)   =  \cos (z \sin \theta) \quad \hbox{and}\quad C_0(w,z) = \frac{\sin (z \sin \theta)}{z \sin \theta}. 
$$

Hence we have obtained the following
\begin{thm}[Clifford-Fourier kernel, $m=2$]
The kernel of the Clifford-Fourier kernel is given by
\[
K_-(x,y) = e^{i \frac{\pi}{2}\Gamma_{\uy}} e^{-i \langle \ux,\uy \rangle} = \cos t + \left(\ux \wedge \uy \right) \frac{\sin t}{t}
\]
with $t= |\ux \wedge \uy| =   \sqrt{|\ux|^{2} |\uy|^{2}-\langle \ux, \uy \rangle^2}$.
\label{KernelDim2}
\end{thm}

Note that the same result has also been obtained in \cite{MR2283868}, although the proof given there is completely different.

\subsubsection{The case $m>2$}

In the case $m=4$ (or $\lambda=1$) we have
\begin{align*}
A_{1}(w,z) & =  2   z^{-1} \sum_{k=0}^{\infty}  J_{2k+ 1}(z) \,  C_{2k}^{1}(w),\\
B_{1}(w,z) &= -  2  z^{-1} \sum_{k=0}^{\infty}(2k+ 2)  J_{2k+2}(z) \, C_{2k+1}^{1}(w).
\end{align*}
Using $ 2 v J_{v}(z) =z(J_{v-1}(z) + J_{v+1}(z))$ and $2 w C^{1}_{k}(w) = C^{1}_{k-1}(w) + C^{1}_{k+1}(w)  $ we can rewrite $B_{1}$ as 
\begin{align*}
B_{1}(w,z)&= -  \sum_{k=0}^{\infty}( J_{2k+1}(z) + J_{2k+3}(z)) \; C_{2k+1}^{1}(w)\\
&=-  \sum_{k=0}^{\infty} J_{2k+1}(z)\, C_{2k+1}^{1}(w) -  \sum_{k=0}^{\infty}  J_{2k+3}(z) \, C_{2k+1}^{1}(w)\\
&= -  \sum_{k=0}^{\infty} J_{2k+1}(z)\, \left( C_{2k+1}^{1}(w) + C_{2k-1}^{1}(w)\right) \\
&= - 2w \sum_{k=0}^{\infty} J_{2k+1}(z) \, C_{2k}^{1}(w)\\
&= - z w A_{1}(w,z).
\end{align*}

Using the recursion relations for $A$, $B$ and $C$ obtained in Lemma \ref{RecursionsProps} we subsequently have
\begin{align*}
A_{k} =&\, (-1)^{k-1} \, k \,  z^{-k-1} \partial_{w}^{k-1} \left( w^{-1}\partial_{w} B_{0} \right),\\
B_{k} =&\, (-1)^{k} z^{-k} \partial_{w}^{k} B_{0}, \\
C_{k} =&\, (-1)^{k} z^{-k-2} \partial_{w}^{k} \left( w^{-1}\partial_{w} B_{0} \right).
\end{align*}

As we have already calculated $B_{0}$ in the previous subsection, we have hence obtained the following theorem.

\begin{thm} \label{thm:even_m}
The kernel of the Clifford-Fourier transform in even dimension $m \geq 2$ is given by
\begin{align*}
K_-(x,y) & = e^{i \frac{\pi}{2}\Gamma_{\uy}} e^{-i \langle \ux,\uy \rangle} \\
& = A_{(m-2)/2}(z,w) +B_{(m-2)/2}(z,w)+ (\ux \wedge \uy) \ C_{(m-2)/2}(z,w),
\end{align*}
where  $z = |\ux||\uy|$, $w = \langle \xi,\eta \rangle$ and
\begin{align*}
A_{(m-2)/2}(z,w) &=(-1)^{m/2} \, \frac{m-2}{2} \,  z^{-\frac{m}{2}-2} \partial_{w}^{\frac{m}{2}-2} \left( w^{-1}\partial_{w} \cos (z \sqrt{1-w^{2}}) \right),\\
B_{(m-2)/2}(z,w)&= -(-1)^{m/2} z^{1-\frac{m}{2}} \partial_{w}^{\frac{m}{2}-1} \cos (z \sqrt{1-w^{2}}),\\
C_{(m-2)/2}(z,w) &= -(-1)^{m/2} z^{-1-\frac{m}{2}} \partial_{w}^{\frac{m}{2}-1} \left( w^{-1}\partial_{w} \cos (z \sqrt{1-w^{2}}) \right).
\end{align*}
\end{thm}

By using variables $s = z w$ and $t = z \sqrt{1-w^2}$, we can carry out the differentiation explicitely. The
result is a completely explicit formula for the kernel of the Clifford-Fourier transform in terms of a finite
sum of Bessel functions. This is the subject of the following theorem.

\begin{thm}[$m$ even]  \label{EvenExplicit}
The kernel of the Clifford-Fourier transform in even dimension $m>2$ is given by
\begin{align*}
K_-(x,y) &= e^{i \frac{\pi}{2}\Gamma_{\uy}} e^{-i \langle \ux,\uy \rangle}\\
& = (-1)^{\frac{m}{2}} \left(\frac{\pi}{2} \right)^{\frac{1}{2}} \left(A_{(m-2)/2}^{*}(s,t) +B_{(m-2)/2}^{*}(s,t)+ (\ux \wedge \uy) \ C_{(m-2)/2}^{*}(s,t)\right)
\end{align*}
where $s=\langle \ux,\uy \rangle$ and $t= |\ux \wedge \uy| =   \sqrt{|\ux|^{2} |\uy|^{2}-s^2}$ and
\begin{align*}
A_{(m-2)/2}^{*}(s,t) &=\sum_{\ell=0}^{\left\lfloor \frac{m}{4}-\frac{3}{4}\right\rfloor}  s^{m/2-2-2\ell} \ \frac{1}{2^{\ell} \ell!} \ \frac{\Gamma \left( \frac{m}{2} \right) }{ \Gamma \left( \frac{m}{2}-2\ell -1 \right)} \widetilde{J}_{(m-2\ell-3)/2}(t),\\
B_{(m-2)/2}^{*}(s,t)&= - \sum_{\ell=0}^{\left\lfloor \frac{m}{4}-\frac{1}{2}\right\rfloor} s^{m/2-1-2\ell} \ \frac{1}{2^{\ell} \ell!} \ \frac{\Gamma \left( \frac{m}{2} \right)}{\Gamma \left( \frac{m}{2}-2\ell \right)} \widetilde{J}_{(m-2\ell-3)/2}(t),\\
C_{(m-2)/2}^{*}(s,t) &= - \sum_{\ell=0}^{\left\lfloor \frac{m}{4}-\frac{1}{2}\right\rfloor} s^{m/2-1-2\ell} \ \frac{1}{2^{\ell} \ell!} \ \frac{\Gamma \left( \frac{m}{2} \right)}{\Gamma \left( \frac{m}{2}-2\ell \right)}  \widetilde{J}_{(m-2\ell-1)/2}(t)
\end{align*}
with $\widetilde{J}_{\alpha}(t) = t^{-\alpha} J_{\alpha}(t)$.
\end{thm}

\begin{proof}
The proof is carried out by induction. It is easy to check that the formulas are correct for $m=4$. Using Lemma \ref{RecursionsProps}, we need to prove that e.g. $A^{*}_{m/2} = \frac{m/2}{m/2 -1} z^{-1} \partial_{w}A^{*}_{m/2-1}$. Indeed, we check this for the case $m= 4 p$. Then, using the properties
\begin{eqnarray*}
z^{-1} \partial_{w} \widetilde{J}_{\alpha}(t) = s \widetilde{J}_{\alpha+1}(t) \quad \hbox{and} \quad 
z^{-1} \partial_{w} s^{\alpha} = \alpha s^{\alpha-1}
\end{eqnarray*}
we can calculate
\begin{align*}
z^{-1} \partial_{w}A^{*}_{2p-1} = &\,  z^{-1} \partial_{w} \sum_{\ell=0}^{ p-1} s^{2p-2-2\ell} \ \frac{1}{2^{\ell} \ell!} \ \frac{\Gamma \left( 2 p \right)}{\Gamma \left( 2p-2\ell -1 \right)} \widetilde{J}_{(4p-2\ell-3)/2}(t)\\
= &\, \sum_{\ell=0}^{ p-1} s^{2p-1-2\ell} \ \frac{1}{2^{\ell} \ell!} \ \frac{\Gamma \left( 2 p \right)}{\Gamma \left( 2p-2\ell -1 \right)} \widetilde{J}_{(4p-2\ell-1)/2}(t)\\
&+ \sum_{\ell=0}^{ p-2} s^{2p-3-2\ell} \ \frac{1}{2^{\ell} \ell!} \ \frac{ \Gamma \left( 2 p \right)}{\Gamma \left( 2p-2\ell -2 \right)} \widetilde{J}_{(4p-2\ell-3)/2}(t)\\
= &\, (2p-1) s^{2p-1}\widetilde{J}_{(4p-1)/2}(t)\\
& + \sum_{\ell=1}^{ p-1} s^{2p-1-2\ell}\widetilde{J}_{(4p-2\ell-1)/2}(t) \ \frac{1}{2^{\ell} \ell!} \ \frac{\Gamma \left( 2 p \right)}{\Gamma \left( 2p-2\ell  \right)} (2p  -1) \\
= &\, \frac{2p-1}{2p} A^{*}_{2p}.
\end{align*}
The other cases are treated similarly.
\end{proof}

\subsection{The case $m$ odd}
\label{CFodd}

Using Lemma \ref{RecursionsProps} we immediately obtain an analog of Theorem \ref{thm:even_m}.

\begin{thm}\label{KernelOdd}
The kernel of the Clifford-Fourier transform in odd dimension $m \geq 3$ is given by
\[
K_-(x,y) =  A_{(m-2)/2}(z,w) +B_{(m-2)/2}(z,w)+ (\ux \wedge \uy) \ C_{(m-2)/2}(z,w)
\]
where  $z = |\ux||\uy|$ and $w = \langle \xi,\eta \rangle$ and
\begin{align*}
A_{(m-2)/2}(z,w) &= (-1)^{(m-3)/2} \frac{m-2}{2} z^{(3-m)/2} \partial_{w}^{(m-3)/2} A_{1/2}\\
B_{(m-2)/2}(z,w) &=   (-1)^{(m-3)/2} z^{(3-m)/2} \partial_{w}^{(m-3)/2} B_{1/2}\\
C_{(m-2)/2}(z,w) &= (-1)^{(m-1)/2}  z^{(1-m)/2} \partial_{w}^{(m-1)/2}  A_{1/2}.
\end{align*}
\end{thm}
In other words, it suffices to determine $A_{1/2}$ and $B_{1/2}$, which are given by (see Theorem \ref{seriesthm})
\begin{align*}
A_{1/2}(w,z) &= \left(\frac{\pi}{2z}\right)^{1/2} \frac{1}{2} \sum_{k=0}^{\infty}  (-i + (-1)^{k})  J_{k+ 1/2}(z) \;  C_{k}^{1/2}(w),\\
B_{1/2}(w,z) &=  \left(\frac{\pi}{2z}\right)^{1/2}\sum_{k=0}^{\infty}(k+ 1/2)  (i + (-1)^{k})  J_{k+ 1/2}(z) \; C_{k}^{1/2}(w).
\end{align*}
We can rewrite their sum, using $(-1)^{k} C_{k}^{1/2}(w) =  C_{k}^{1/2}(-w)$, as 
$$
  A_{1/2}(w,z) + B_{1/2}(w,z) = \left(\frac{\pi}{2z}\right)^{1/2} \left( U(w,z) + V(-w,z) + i V(w,z) \right),
$$ 
where 
\begin{align*}
 U(w,z) & =  \sum_{k=0}^\infty (-1)^k J_{k+1/2}(z) P_k(w), \\ 
 V(w,z) &  = \sum_{k=0}^\infty  k J_{k+1/2}(z) P_k(w)
\end{align*}
with $P_{k}(w)= C_{k}^{1/2}(w)$ being the Legendre polynomials of degree $k$. We then obtain the following integral representation for $U$ and $V$.

\begin{lem} \label{OddIntegrals}
With $w= \cos \t$ one has
\begin{align*}
 U(w,z)  & =\left(\frac{z}{2\pi}\right)^{1/2}  \int_{-1}^1
  e^{i z u}  e^{- \frac{z}{2} (1-u^2)\cos \t} J_0\left(\frac{z}{2} (1-u^2) \sin\t \right)du   \\
 V(w,z) &  =  \frac{1}{\sqrt{\pi}} \left(\frac{z}{2}\right)^{3/2} \int_{-1}^1
  e^{i z u}  e^{\frac{z}{2} (1-u^2)\cos \t} (1-u^2) \\
  &\quad  \times \left[  \cos \t J_0\left(\frac{z}{2} (1-u^2) \sin \t \right) - \sin \t J_1\left(\frac{z}{2} (1-u^2) \sin \t \right)
     \right] du. 
\end{align*}
\end{lem} 

\begin{proof}
We need to recall the generating function for the Legendre polynomials $P_n(w)$
(see \cite[(4.10.7)]{Sz})
\begin{equation}\label{Legendre}
   \sum_{n=0}^\infty \frac{P_n(\cos \t)}{n!} r^n = e^{r \cos \t} J_0(r \sin \t), 
\end{equation}
which implies, upon taking derivative on $r$ and using $J_0'(z) =- J_1(z)$, 
\begin{align}\label{Legendre2}
   \sum_{n=1}^\infty n \frac{P_n(\cos \t)}{n!} r^{n-1} = e^{r \cos \t} \left[ \cos \t J_0(r \sin \t)
       - \sin \t J_1(r\sin \t)\right]. 
\end{align}
We also recall the integral representation 
for the Bessel function (\cite[p. 81]{Er} or \cite[(1.71.6)]{Sz}) 
\begin{equation} \label{Bessel}
   J_\a(z) = \frac{1}{\Gamma(\a+\frac12)\Gamma(\frac12)} \left(\frac{z}{2} \right)^\a \int_{-1}^1 
     e^{i z u} (1-u^2)^{\a-\frac12} du.  
\end{equation}
Now, combining \eqref{Legendre} with \eqref{Bessel}, we obtain readily with $w = \cos \t$, 
\begin{align*}
  \sum_{k=0}^\infty J_{k+1/2}(z) (-1)^k P_k(w) & =  \sqrt{\frac{z}{2\pi}}  \int_{-1}^1
  e^{i z u} \sum_{k=0}^\infty \frac{P_k(\cos \t)}{k!} \left(\frac{- z}{2} (1-u^2)\right)^{k}  du \\   
& =\sqrt{\frac{z}{2\pi}}  \int_{-1}^1
  e^{i z u}  e^{- \frac{z}{2} (1-u^2)\cos \t} J_0\left(\frac{z}{2} (1-u^2) \sin\t \right)du, 
\end{align*}
where we have used $J_0(-z) = J_0(z)$.  Furthermore, using \eqref{Legendre2} we obtain
\begin{align*}
  \sum_{k=0}^\infty & k J_{k+1/2}(z) P_k(w)   =  \left(\frac{z}{2\pi}\right)^{1/2}  \int_{-1}^1
  e^{i z u} \sum_{k=0}^\infty k \frac{P_k(\cos \t)}{k!} \left(\frac{z}{2} (1-u^2)\right)^{k}  du \\   
& =\frac{1}{\sqrt{\pi}} \left(\frac{z}{2}\right)^{3/2} \int_{-1}^1
  e^{i z u}  e^{\frac{z}{2} (1-u^2)\cos \t} (1-u^2) \\
  &\quad  \times \left[  \cos \t J_0\left(\frac{z}{2} (1-u^2) \sin \t \right) - \sin \t J_1\left(\frac{z}{2} (1-u^2)
   \sin \t \right)
     \right] du 
\end{align*}
and complete the proof. 
\end{proof}

\begin{rem}
One obvious question is whether it is possible to write $U$ and $V$ in terms of sums in Bessel functions, as
in Theorem \ref{EvenExplicit}, but we do not know if this is possible. Moreover, the integral formulas in the lemma do not seem to yield an upper bound, as in the subsequent Lemma \ref{lem:bound}, of the kernel $K(x,y)$ for odd $m$.  
\end{rem}

\section{Further properties of the kernel function}
\label{secFurtherprops}
\setcounter{equation}{0}

From the explicit expression for the kernel of the Clifford-Fourier transform, we can derive several properties of the kernel. We start with a simple observation. 

\begin{prop} \label{KernelMultiplication}
Let $m=2$. Then the kernel of the Clifford-Fourier transform satisfies
\[
K_-(x,z) K_-(y,z) = K_-(x+y,z).
\]
If the dimension $m$ is even and $m>2$ then 
\[
K_-(x,z) K_-(y,z) \neq K_-(x+y,z).
\]
\end{prop}

\begin{proof}
If $m=2$, then the kernel, given in Theorem \ref{KernelDim2}, can be rewritten as
\begin{equation} \label{ker2simple}
K_-(x,y) = \cos{(x_{1}y_{2}-x_{2}y_{1})} + e_{1}e_{2} \sin{(x_{1}y_{2}-x_{2}y_{1})}. 
\end{equation}
The result then follows using basic trigonometric identities. This result is also obtained in \cite[Proposition 5.1]{MR2283868}.

Now suppose $m$ even and $m>2$. We consider co-ordinates $x_{1}\neq 0$ and $x_{2} =\ldots= x_{m}= 0$ and similar for the $y$ and $z$ variables. Then the explicit formula for the kernel, given in Theorem \ref{EvenExplicit}, reduces to
\[
K_-(x,y) = \sum^{m/2-1}_{j=0} a_{j} (x_{1}y_{1})^{j}, \quad a_{j} \in \mR
\]
and similar for $K_-(y,z)$ and $K_-(x+y,z)$. We readily observe that $K_-(x,z) K_-(y,z)$ is a polynomial of higher degree than $K_-(x+ y,z)$, hence the equality cannot hold.
\end{proof}

\begin{rem}
The result in Proposition \ref{KernelMultiplication} for $m=2$ is not really surprising. As $(e_{1}e_{2})^{2}=-1$, formula (\ref{ker2simple}) implies that, upon substituting $e_{1}e_{2}$ by the imaginary unit $i$, the kernel is equal to the kernel of the classical Fourier transform. This is clearly not the case for higher even dimensions.
\end{rem}

The explicit formula allows us to study the (un)boundedness of the kernel. We start with the following lemma.

\begin{lem} \label{lem:bound}
Let $m$ be even. For $ \ux, \uy \in \RR^m$, there exists a constant $c$ such that
\begin{align*}
|A_{(m-2)/2}^{*}(s,t) + B_{(m-2)/2}^{*}(s,t)| &\leq  c (1 + |\la \ux, \uy \ra |)^{(m-2)/2},  \\
|(x_{j} y_{k} - x_{k}y_{j}) C_{(m-2)/2}^{*}(s,t)| &\leq  c (1 + |\la \ux, \uy \ra |)^{(m-2)/2}, \qquad
 j \neq k. 
\end{align*}
\end{lem}

\begin{proof}
We work with the explicit formula of $K_-(x,y)$ in Theorem 3 and use the integral representation 
for the Bessel function (see formula (\ref{Bessel})). This  implies immediately that 
$$
    |  z^{-\a} J_\a(z) | \le c, \qquad z \in  \RR. 
$$
By the explicit formula of $A^{*}_\l$ and $B^{*}_\l$, it follows readily that 
\begin{align*}
    | A_{(m-2)/2}^{*}(s,t)| & \le c \sum_{\ell =0}^{\lfloor  \frac{m}{4} - \frac34\rfloor} |s|^{m/2-2-2\ell} \le c (1 + |s|)^{m/2-2} \\
    | B_{(m-2)/2}^{*}(s,t)| & \le c \sum_{\ell =0}^{\lfloor  \frac{m}{4} - \frac12\rfloor} |s|^{m/2-1-2\ell} \le c (1 + |s|)^{m/2-1},
\end{align*}
so that $|A_{(m-2)/2}^{*}(s,t) + B_{(m-2)/2}^{*}(s,t)|$ has the desired bound. Furthermore, integrating by parts
in \eqref{Bessel} shows that for $\a > 1/2$,
$$
   J_\a(z) =  \frac{-1}{\Gamma(\a+\frac12)\Gamma(\frac12)} \left(\frac{z}{2} \right)^\a
        \frac{ 2 \a -1 }{iz} \int_{-1}^1 e^{i u z} u (1-u^2)^{\a-3/2} du, 
$$
from which it follows readily that 
$$
       |  z^{-\a + 1} J_\a(z) | \le c, \qquad z \in  \RR. 
$$
Since $t = \sqrt{-(x \wedge y)^{2}}$ and using (\ref{estimateCterm}), it follows then that 
\begin{align*}
 |(x_{j} y_{k} - x_{k}y_{j}) C^{*}_{(m-2)/2}(s,t) | & \le  | t \, C^{*}_{(m-2)/2}(s,t)| \\
   & \le c \sum_{\ell =0}^{\lfloor  \frac{m}{4} - \frac12\rfloor} |s|^{m/2-1-2\ell} 
          \le c (1 + |s|)^{m/2-1}.
\end{align*}
This completes the proof. 
\end{proof}  

Recall that the kernel $K_-(x,y)$ is a Clifford algebra valued function. It can be decomposed as
\begin{equation}
K_-(x,y) = K_{0}^{-}(x,y) + \sum_{i<j} e_{i} e_{j} K_{ij}^{-}(x,y)
\label{DecompKernel}
\end{equation}
with $ K_{0}^{-}(x,y)$ and $K_{ij}^{-}(x,y)$ scalar functions. Now we immediately have the following bounds.

\begin{thm} \label{kernelBound}
Let $m$ be even. For $x,y \in \RR^m$, one has
\begin{align*}
| K_{0}^{-}(x,y)| &\leq  c (1+|\ux|)^{(m-2)/2}(1+|\uy|)^{(m-2)/2},  \\
| K_{ij}^{-}(x,y)| &\leq c (1+|\ux|)^{(m-2)/2}(1+|\uy|)^{(m-2)/2}, \qquad
 j \neq k. 
\end{align*}
\end{thm}

\begin{proof} 
This follows immediately from Theorem \ref{EvenExplicit}, Lemma \ref{lem:bound}, $|\la \ux , \uy \ra| \le |\ux|\cdot |\uy|$ and the elementary inequality 
$1+ | \ux| \cdot |\uy| \le (1+|\ux|) (1+|\uy|)$.  
\end{proof}

Since the integral representation of $J_\a$ also shows that $\wt J_\a(0)$ is a constant, we see that the
order $(m-2)/2$ in the upper bound is sharp. Note that the kernel is bounded if $m =2$, which has 
already been observed in \cite{MR2283868}. 

The bound of the kernel function defines the domain of the Clifford-Fourier transform, see Theorem 
\ref{CFdomain} below. 
In the following section we will prove that $\CF_{\pm}$ maps  $\cS(\RR^m) \otimes \cC l_{0,m}$ continuously to 
$\cS(\RR^m) \otimes \cC l_{0,m}$, for which the following properties of the kernel will be instrumental. 
Recall that $\partial_{\ux}$, defined  in \eqref{DiracOperator}, is the Dirac operator and $\Delta$ is the Laplacian. 

\begin{prop} \label{DiffKernel}
For all $m$, the kernel $K_\pm(x,y)$ satisfies the properties
\begin{align*}
\upy K_\pm(x,y) &= \pm (\mp i)^{m} K^{\mp}(x,y) \ux,\\
\Delta_{y} K_\pm(x,y)&= - |\ux|^{2} K_\pm(x,y)
\end{align*}
and
\begin{align*}
\uy K_\pm(x,y) &= \mp (\mp i)^{m}  \left( K^{\mp} (x,y) \upx \right),\\
\Delta_{x} K_\pm (x,y)&= - |\uy|^{2} K_\pm(x,y),
\end{align*}
where $\Delta_{x}$ means that the Laplacian $\Delta$ is acting on the $x$ variables and with
\[
\left( K^{\mp} (x,y) \upx \right) = \sum_{i=1}^{m} \left(\partial_{x_{i}} K^{\mp} (x,y) \right) e_{i}
\]
the action of the Dirac operator on the right.
\end{prop}

\begin{proof}
Using Theorem \ref{ospFamily}, it is easy to obtain
\begin{align*}
\upy \Gamma_{\uy}^{k}&= (m-1-\Gamma_{\uy})^{k} \upy\\
\uy \Gamma_{\uy}^{k}&= (m-1-\Gamma_{\uy})^{k} \uy
\end{align*}
for $k \in \mN$. Taking into account that, by definition, $K_\pm(x,y) = e^{\mp i \frac{\pi}{2}\Gamma_{\uy}} e^{-i \langle \ux,\uy \rangle}$, we subsequently calculate
\begin{align*}
\upy K_+(x,y) &=\upy e^{- i \frac{\pi}{2}\Gamma_{\uy}} e^{-i \langle \ux,\uy \rangle}\\
&=e^{- i \frac{\pi}{2}\left(m-1-\Gamma_{\uy}\right)} \upy  e^{-i \langle \ux,\uy \rangle}\\
&=(-i)^{m-1} e^{ i \frac{\pi}{2}\Gamma_{\uy}} (-i \ux) e^{-i \langle \ux,\uy \rangle}\\
&= (-i)^{m} K_-(x,y) \ux.
\end{align*}
The expression for $\uy K_\pm$ is proven in a similar way. Using $\upx^{2}=-\Delta_{x}$ and $\uy^{2}=-|\uy|^{2}$ we immediately obtain the  other two properties.
\end{proof}

\begin{rem}
We do not know the action of the usual partial derivatives on the kernel except when $m =2$. For $m =2$, a quick computation using 
(\ref{ker2simple}) shows that 
\begin{align*}
\partial_{x_{1}}K_-(x,y) & = y_{2} e_{1}e_{2} K_-(x,y), \\
 \partial_{x_{2}}K_-(x,y) & = - y_{1} e_{1}e_{2} K_-(x,y).
\end{align*}
\end{rem}

\section{Properties of the Clifford-Fourier transform}
\setcounter{equation}{0}
\label{Inversion}

As an immediate consequence of Theorem \ref{kernelBound}, we can now specify the domain in the definition of the Clifford-Fourier transform. Let us define a class of functions
$$
    B(\RR^m) : = \left\{ f\in L^1(\RR^m): \int_{\RR^m} (1+|\uy|)^{(m-2)/2} | f(y)| dy < \infty \right\}.
$$

\begin{thm} \label{CFdomain}
Let $m$ be an even integer. The Clifford-Fourier transform is well-defined on $B(\RR^m) \otimes \cC l_{0,m}$. In particular, for 
$f \in B(\RR^m) \otimes \cC l_{0,m}$, $\cF_\pm f$ is a continuous function.
\end{thm}

\begin{proof}
It follows immediately from Theorem \ref{kernelBound} that the transform is well-defined on $B(\RR^m) \otimes \cC l_{0,m}$. The continuity of $f$ follows from the continuity of the kernel 
and the dominated convergence theorem.
\end{proof}

For $m$ being even we can now establish the inversion formula for Schwartz class functions. First we state a lemma. 

\begin{lem} \label{DiffTransform}
Let $m$ be even and $f \in \cS(\mR^{m})$. Then  
\begin{align*}
\cF_{\pm} \left(\ux \, f \right) & = \mp (-1)^{m/2} \upy \cF_{\mp} \left( f \right), \\
\cF_{\pm} \left(\upx f \right) & = \mp (-1)^{m/2}  \uy \cF_{\mp} \left( f \right).
\end{align*} 
\end{lem}

\begin{proof}
The first identity follows immediately from Proposition \ref{DiffKernel}. Because $m$ is even, the kernel $K_\pm$ has a polynomial bound according to Theorem \ref{kernelBound} and we can apply integration by parts, which gives the second identity by Proposition \ref{DiffKernel}. 
\end{proof}

\begin{thm} \label{CFschwartz} 
Let $m$ be even. Then $\cF_\pm$ is a continuous operator on $\cS(\RR^m) \otimes \cC l_{0,m}$.
\end{thm}

\begin{proof}
Using formula (\ref{DecompKernel}) we can rewrite $\cF_{-}$ as 
\[
\cF_{-} = \cF_{0}^{-} + \sum_{i<j} e_{i} e_{j} \cF_{ij}^{-}
\]
with
\begin{align*}
\cF_{0}^{-}(f)(y) &= (2\pi)^{-m/2} \int_{\mR^{m}} K_{0}^{-}(x,y) f(x)dx\\
\cF_{ij}^{-}(f)(y) &= (2\pi)^{-m/2} \int_{\mR^{m}} K_{ij}^{-}(x,y) f(x) dx.
\end{align*}
Moreover, as e.g. $\Delta_{y} K_\pm(x,y)= - |\ux|^{2} K_\pm(x,y)$, we have immediately that also
\begin{align*}
\Delta_{y} K_{0}^{-}(x,y)&= - |\ux|^{2} K_{0}^{-}(x,y)\\
\Delta_{y} K_{ij}^{-}(x,y)&= - |\ux|^{2} K_{ij}^{-}(x,y).
\end{align*}
and in particular
\begin{align}
\label{propF0}
\begin{split}
\cF_{0}^{-}(\Delta_{x} f)(y) &= -|\uy|^{2}\cF_{0}^{-}(f)(y)\\
\cF_{0}^{-}(|\ux|^{2}f)(y) &= -\Delta_{y} \cF_{0}^{-}(f)(y).
\end{split}
\end{align}
The same results hold for $\cF_{ij}^{-}$.

It clearly suffices to prove that $\cF_{0}^{-}$ and $\cF_{ij}^{-}$ are continuous maps on $\cS(\RR^m)$. We give the proof for $\cF_{0}^{-}$, the other cases being similar.

Recall that the Schwartz class $\cS(\RR^m)$ is endowed with the topology defined by the family of semi-norms 
$$
    \rho_{\a,\b} (f) : = \sup_{x\in \RR^m} |x^\a \partial^\b f(x)|, \qquad \a ,\b \in \NN_0^m,  
$$
and $f \in \cS(\RR^m)$ if $\rho_{\a,\b}(f) < \infty$ for all $\a,\b$. The latter condition, however,
is equivalent to $\rho_{\a,n}^* (f) < \infty$ for 
$$
   \rho_{\a,n}^* (f) : =  \sup_{x\in \RR^m} |\ux^\a \Delta^n f(x)|, \qquad \a \in \NN_0^m, \quad n \in \NN_0.  
$$
Indeed, by induction, it is easy to see that $x^\a \partial^\b f = \sum c_{\g,\delta} \partial^\g (x^\delta f(x))$, where the 
sum is over $\{(\g, \delta): |\gamma| < |\a|, |\delta| \le |\b|\}$ and $c_{\g,\delta}$ are finite numbers, so that we only need to
consider the semi-norms defined by $\|\partial^\a (x^\b f)\|_\infty$, and we also know that (cf. \cite{Timo})
$\|\partial^\a g\|_\infty  \le c (\|\Delta^n g \|_\infty + \|g\|_\infty)$ for $|\a| \le 2 n$.  

Now let $\a \in \NN_0^m$ and $n \in \NN_0$. If $|\uy| \le 1$, then by (\ref{propF0}) and Theorem \ref{kernelBound}, 
\begin{align*}
   |\uy^\a \Delta_y^n \cF_{0}^{-} f(y)|  & =  |\uy^\a |  \cdot | \cF_{0}^{-}( | \{\underline{\cdot }\}|^{2n} f)(y)| \\ 
   &\le c (1+|\uy|)^{(m-2)/2} \int_{\RR^m}  (1+|\ux|)^{(m-2)/2} |\ux|^{2n} |f(x)| dx \\
   & \le c_{1} \sup_{x\in \RR^m}|(1+|\ux|)^{3m/2} |\ux|^{2n} f(x)|
\end{align*}
as $f$ is a Schwartz class function. For $|\uy| \ge 1$ we use $\Delta_{x} K^{-}_{0}(x,y)  = -|\uy|^2 K^{-}_{0}(x,y)$ and partial integration to conclude that 
\begin{align*}
   |\uy^\a \Delta_y^n \cF_{0}^{-} f(y)|  & =  |\uy^\a |  \cdot |\cF_{0}^{-}( | \{\underline{\cdot }\}|^{2n} f)(y)| \\ 
   &= |\uy^\a|\cdot |\uy|^{-2 \s} |\cF_{0}^{-}(\Delta^\s  | \{\underline{\cdot }\}|^{2n} f)(y) | \\
   & \le c |\uy|^{|\a| - 2\s}  (1+|\uy|)^{(m-2)/2} \, | \int_{\RR^m}  (1+|\ux|)^{(m-2)/2} \Delta^\s(|\ux|^{2n} f(x)) dx | \\
   &  \le c_{2} \sup_{x\in \RR^m}|(1+|\ux|)^{3m/2} \Delta^\s(|\ux|^{2n} f(x))|
\end{align*}
if $2 \s \ge |\a| + (m-2)/2$. 
Summarizing, we have
\begin{align*}
\sup_{y\in \RR^m}|\uy^\a \Delta_y^n \cF_{0}^{-} f(y)| \le \max &\left\{c_{1} \sup_{x\in \RR^m}|(1+|\ux|)^{3m/2} |\ux|^{2n} f(x)|, \right.\\
& \left.   c_{2} \sup_{x\in \RR^m}|(1+|\ux|)^{3m/2} \Delta^\s(|\ux|^{2n} f(x))|\right\}.
\end{align*}
This completes the proof.
\end{proof}

\begin{rem}
For $m$ being odd, we do not know if the bound for the kernel still holds. If it does, the above proof clearly carries over. 
\end{rem}

The series expressions for $A_{\l}$, $B_{\l}$ and $C_{\l}$ obtained in Theorem \ref{seriesthm} allow us to study the radial behavior of the Clifford-Fourier transform. This is the subject of the following theorem.

\begin{thm}
Let $M_{\ell} \in \cM_{\ell}$ be a spherical monogenic of degree $\ell$. Let $f(x)= f_0(|\ux|)$ be a real-valued radial function in 
$\cS(\RR^m)$. Further, put $\xi= \ux/|\ux|$, $\eta = \uy/|\uy|$ and $r = |\ux|$. Then one has
\begin{equation} \label{Hankel}
\cF_{-} \left(f_{0}(r)M_{\ell}(x) \right)=  (-1)^{\ell}  M_{\ell}(\eta) \int_{0}^{+\infty} r^{m+\ell-1}f_0(r)   z^{-\l} J_{\ell + \l}(z)  dr
\end{equation}
and
\begin{eqnarray*}
\cF_{-} \left( f_{0}(r) \ux M_{\ell}(x) \right)=  - i^{m}  \eta M_{\ell}(\eta) \int_{0}^{+\infty} r^{m+\ell}f_0(r)   z^{-\l} J_{\ell +1+ \l}(z)  dr
\end{eqnarray*}
with $z= r |\uy|$ and $\l = (m-2)/2$.
\label{RadialBeh}
\end{thm}

\begin{proof}
We prove the first property in the case $\ell$ is even. Then we have, using (\ref{FH})
\begin{align*}
\int_{\mR^{m}} A_{\l}f_0(r) M_{l}(x) d x&= c 2^{\l-1}\Gamma(\l+1)(i^{m}+1) \frac{\l}{\ell+\l}M_{\ell}(\eta)\\
& \times \int_{0}^{+\infty}  r^{m+\ell-1}f_{0}(r)   z^{-\l} J_{\ell + \l}(z) dr
\end{align*}
and
\begin{align*}
\int_{\mR^{m}} B_{\l}f_0(r) M_{l}(x) d x &= -c 2^{\l-1}\Gamma(\l)(i^{m}-1) \l M_{\ell}(\eta)\\
& \times \int_{0}^{+\infty}  r^{m+\ell-1}f_{0}(r)   z^{-\l} J_{\ell + \l}(z) dr.
\end{align*}
The term containing $C_{\lambda}$ is more complicated. We first rewrite
\[
\ux \wedge \uy M_{\ell} = - \uy \ux M_{\ell} - \la \ux, \uy \ra M_{\ell}
\]
and calculate both terms seperately. As $\uy \ux M_{\ell} \in \cH_{\ell+1}$, we obtain using (\ref{Geg1})
\begin{align*}
& \int_{\mR^{m}} C_{\l}f_0(r) \uy \ux M_{l}(x) d x \\
& \qquad  = c 2^{\l}\Gamma(\l+1) (i^{m}+1) M_{\ell}(\eta) \int_{0}^{+\infty}  r^{m+\ell-1}f_{0}(r)   \sum_{k=\ell/2 +1}^{\infty} z^{-\l} J_{2k + \l}(z) dr. 
\end{align*}
To calculate the second term, we first apply (\ref{Geg2}) followed by (\ref{Geg1}), yielding
\begin{align*}
&\int_{\mR^{m}} C_{\l}f_0(r) \la \uy, \ux \ra M_{l}(x) d x\\
&\quad = - c 2^{\l}\Gamma(\l+1) (i^{m}+1) M_{\ell}(\eta) \int_{0}^{+\infty} r^{m+\ell-1}f_{0}(r)    \sum_{k=\ell/2 +1}^{\infty} z^{-\l} J_{2k + \l}(z) dr\\
& \qquad - c \frac{\ell}{\ell+\l}2^{\l-1}\Gamma(\l+1) (i^{m}+1) M_{\ell}(\eta) \int_{0}^{+\infty}  r^{m+\ell-1}f_{0}(r)    z^{-\l} J_{\ell + \l}(z) dr. 
\end{align*}
Collecting all terms then gives the desired result. 

The other cases are treated in a similar way.
\end{proof}

When $\ell =0$, the transform \eqref{Hankel} is, up to a constant, the Hankel transform $H_\l$ defined
by \cite[p. 456]{MR0010746} 
\begin{equation} \label{Hankl}
   H_\lambda f(s) : =  \int_0^\infty f(r) \frac{J_\lambda(rs)}{(rs)^\lambda} 
     r ^{2\lambda +1} dr 
\end{equation}
for $\lambda > -1/2$. The inverse Hankl transform is given by 
\begin{equation} \label{inversHankl}
     f(s) : =  \int_0^\infty H_\lambda f(r) \frac{J_\lambda(rs)}{(rs)^\lambda} 
     r ^{2\lambda +1} dr,
\end{equation}
which holds under mild conditions on $f$. As a consequence, we have the following corollary: 

\begin{cor}
If $f(x) = f_0(|\ux|)$, then $\cF_{-} f(x) =  H_\lambda f_0(|\ux|)$. 
\end{cor}

In particular, this shows that $\cF_{-}$ coincides with the classical Fourier 
transform for radial functions. This is as expected, since $\Gamma_{\ux}$ commutes with radial functions. 

Another corollary is the action of $\cF_{-}$ on the basis $\{ \psi_{j,k,l}\}$ defined in \eqref{basis}. 

\begin{thm} \label{EigenfunctionsKernel}
For the basis $\{\psi_{j,k,l}\}$ of  $\cS(\mR^{m}) \otimes \cC l_{0,m}$ , one has
\begin{align*}
\cF_{\pm}(\psi_{2j,k,l}) &= (-1)^{j+k} (\mp 1)^{k} \psi_{2j,k,l},\\
\cF_{\pm} (\psi_{2j+1,k,l}) &= i^{m} (-1)^{j+1} (\mp 1 )^{k+m-1} \psi_{2j+1,k,l}.
\end{align*}
In particular, the action of $\CF_\pm$ coincides with the operator $e^{ \frac{i \pi m}{4}} e^{\mp \frac{i \pi}{2}\Gamma  }e^{\frac{i \pi}{4}(\Delta - |\ux|^{2})}$ when restricted to the basis $\{ \psi_{j,k,l}\}$ and 
\begin{equation}\label{inversion}
\cF_{\pm}^{-1} \cF_{\pm} = Id 
\end{equation}
on the basis $\{\psi_{j,k,l}\}$, with $\cF_{\pm}^{-1}$ as in definition \ref{def:FC}. Moreover, when $m$ is even, \eqref{inversion} holds for all $f \in \cS(\RR^m) \otimes \cC l_{0,m}$.
\end{thm}

\begin{proof}
For $\cF_{-}$ this follows from the explicit expression of $\{ \psi_{j,k,l}\}$ (see formula \eqref{basis}), Theorem \ref{RadialBeh} and the following identity (see \cite[exercise 21, p. 371]{Sz})
\[
\int_{0}^{+\infty} r^{2\l+1} (rs)^{-\l} J_{k+\l}(rs)\, r^{k} L_{j}^{k+\l}(r^{2}) e^{-r^{2}/2}dr = (-1)^{j}s^{k} L_{j}^{k+\l}(s^{2}) e^{-s^{2}/2}, 
\]
where $L_j^\a$ is the Laguerre polynomial. The proof for $\cF_{+}$ is similar and the resulting eigenvalues clearly coincide with the ones given in formula (\ref{EigCF1}). Using Proposition \ref{kernelsrels} and definition \ref{def:FC}, one computes the eigenvalues of $\cF_{\pm}^{-1}$ on the basis $\{ \psi_{j,k,l}\}$ in a similar way. Formula (\ref{inversion}) then immediately follows.

For $m$ being even, the final result follows from Theorem \ref{CFschwartz} and the fact that 
$\{\psi_{j,k,l}\}$ is a dense subset of  $\cS(\mR^{m}) \otimes \cC l_{0,m}$. 
\end{proof}

\section{Generalized translation operator}
\label{SecTranslation}
\setcounter{equation}{0}

The convolution $f * g$ plays a fundamental role in classical Fourier analysis. It is defined by
$$
   (f * g) (x) = \int_{\RR^m} f(y) g(x-y) dy,
$$ 
and it depends on the translation operator $\tau_y: f  \mapsto f(\cdot - y)$. Under the Fourier transform,
$\tau_y$ satisfies $\wh {\tau_y f} (x) = e^{- i \la x,y\ra} \wh f(x)$, $x \in \RR^m$. We define a generalized 
translation operator related to the Clifford-Fourier transform. 

\begin{defn}
Let $f \in \cS(\RR^m) \otimes \cC l_{0,m}$. For $y \in \RR^m$ the generalized translation operator $f \mapsto \tau_y f$ is defined by 
$$
    \CF_- {\tau_y f}(x) = K_- (y, x)  \CF_- f(x), \qquad x \in \RR^m. 
$$
\end{defn}
  
By Theorem \ref{EigenfunctionsKernel}, this operator is well-defined when $m$ is even and it can be expressed, by the inversion of 
$\cF_-$, as an integral operator
\begin{equation} \label{translation1}
  \tau_y f(x)  =  (2 \pi)^{-\frac{m}{2}} \int_{\RR^m} \overline{K_-(\xi,x)}K_-(y,\xi) \CF_- f(\xi)d\xi. 
\end{equation}
 
We should emphasize at this point that functions taking values in the Clifford algebra $\cC l_{0,m}$ do not commute,
i.e., $f g \ne gf$, in general. As a result, it is not clear what properties can be established for the 
generalized translation operator that we just defined. We start with an observation. 

\begin{prop}
Let $m=2$. Then for all functions $f \in \cS(\mR^{m}) \otimes \cC l_{0,m}$ one has
\[
\tau_{y} f(x) = f(x-y).
\] 
If the dimension $m$ is even and $m>2$ then in general
\[
   \tau_{y} f(x) \neq f(x-y).
\]
\end{prop}

\begin{proof}
This follows immediately from Proposition \ref{KernelMultiplication}. 
\end{proof}

In the case $m =2$, the explicit formula of the kernel function shows that one can identify the Clifford-Fourier transform with the ordinary Fourier transform, so that the above result is not surprising. In the case of $m > 2$, the above proposition suggests that the generalized translation is something new. However, our main result below shows, rather surprisingly, that $\tau_y$ coincides with the classical translation operator if $f$ is a radial function. 

\begin{thm} \label{transl}
Let $f \in \cS(\RR^m)$ be a radial function on $\RR^m$, $f(x) = f_0(|\ux|)$ with $f_0: \RR_+ \mapsto \RR$, then 
$\tau_y f(x) = f_0(|\ux-\uy|)$. 
\end{thm}

The proof of this theorem is long. The key ingredient is a compact formula for the integral 
\[
\int_{\mS^{m-1}} \overline{K_-( r \eta,x)} K_-(y,r \eta)  d\omega(\eta)  
\]
where $\int_{\mS^{m-1}} d\omega(\eta)=1$, which is derived using the series representation of the kernel function. The computation is divided into several auxiliary lemmas. In all these lemmas $k, l$ are natural numbers and $\l = (m-2)/2$. We also use 
$\ux' = \ux / |\ux|$ and $\uy' = \uy / |\uy|$. 

\begin{lem} \label{lemmaInnerprodGeg}
Put $I_{k} = \{k, k-2, k-4, \ldots \}$. Then
\begin{align*}
&\int_{\mS^{m-1}} \la  \eta, \uy' \ra C^{\l+1}_{k}(\la  \eta, \uy' \ra) C^{\l}_{l}(\la  \eta, \ux' \ra) d \omega(\eta)\\
& =\left\{ \begin{array}{ll} 0& l-1 \not \in I_{k}\\ \frac{k+1}{2(k+1+\l)} C^{\l}_{l}(\la  \uy', \ux' \ra)& l-1=k\\ C^{\l}_{l}(\la  \uy', \ux' \ra)& l+1 \in I_{k}. \end{array} \right.\end{align*}
\end{lem}

\begin{proof}
Expand $\la  \eta, \uy' \ra C^{\l+1}_{k}(\la  \eta, \uy' \ra)$ using formula (\ref{Geg2}). Then apply (\ref{Geg1}) and use the reproducing property of (\ref{FH}). This yields the result.
\end{proof}

\begin{lem} \label{IntegralCliffordGeg}
Put $I_{k} = \{k, k-2, k-4, \ldots \}$. Then
\begin{align*}
&\int_{\mS^{m-1}} \eta  C^{\l+1}_{k}(\la  \eta, \uy' \ra) C^{\l}_{l}(\la  \eta, \ux' \ra) d \omega(\eta)\\
& = \begin{cases} \uy' C^{\l}_{l}(\la  \uy', \ux' \ra) & l+1 \in I_{k} \\ \uy' \left[ \frac{l}{2l+m-2} C^{\l}_{l}(\la  \uy', \ux' \ra)  + \frac{2\l}{2l+m-2} \ux' \wedge \uy' C^{\l+1}_{l-1}(\la  \uy', \ux' \ra)  \right]& l-1=k \\ 0 & l-1 \not \in I_{k}. \end{cases} 
\end{align*}
\end{lem}

\begin{proof}
We start by decomposing $C^{\l}_{l}(\la  \eta, \ux' \ra)$ in monogenic components. We have $C^{\l}_{l}(\la  \eta, \ux' \ra) = F_{l} + G_{l}$, with $F_{l} \in \cM_{l}$ and $G_{l} \in \eta \cM_{l-1}$ given by (see the proof of Lemma \ref{actionGamma})
\begin{align*}
F_{l} &= \left(1- \frac{l}{2l+m-2} \right) C^{\l}_{l}(\la  \eta, \ux' \ra)  - \frac{2\l}{2l+m-2} \ux' \wedge \eta \,  C^{\l+1}_{l-1}(\la  \eta, \ux' \ra)\\
G_{l} &= \frac{l}{2l+m-2} C^{\l}_{l}(\la  \eta, \ux' \ra)  + \frac{2\l}{2l+m-2} \ux' \wedge \eta \,  C^{\l+1}_{l-1}(\la  \eta, \ux' \ra). 
\end{align*}
Note that $\eta F_{l}$ is the restriction to the unit sphere of an element of $\cH_{l+1}$ and $\eta G_{l}$ the restriction to the unit sphere of an element of $\cH_{l-1}$ (because $\eta^{2}=-1$). Using (\ref{Geg1}) we can now apply (\ref{FH}), yielding the lemma.
\end{proof}

\begin{lem} \label{IntegralWedge}
One has
\begin{align*}
&\int_{\mS^{m-1}} (\uy' \wedge \eta )  C^{\l+1}_{k}(\la  \eta, \uy' \ra) C^{\l}_{l}(\la  \eta, \ux' \ra) d \omega(\eta) \\
&= -\frac{\l}{k+1+\l}\delta_{l-1, k} (\ux' \wedge \uy' )  C^{\l+1}_{k}(\la  \ux', \uy' \ra).
\end{align*}
\end{lem}

\begin{proof}
Note that $\uy' \wedge \eta  = \uy'  \eta  + \la \eta, \uy' \ra$ and that $\uy'^{2}=-1$. The lemma then follows, using Lemma \ref{lemmaInnerprodGeg} and \ref{IntegralCliffordGeg}.
\end{proof}

\begin{lem} \label{DimRedInt}
Put $k \leq l$. Then
\[
\int_{\mS^{m-1}}  C^{\l+1}_{k}(\la  \eta, \uy' \ra) C^{\l+1 }_{l}(\la  \eta, \ux' \ra) d \omega(\eta) = C^{\l+1}_{k}(\la  \ux', \uy' \ra), 
\]
provided $k+l$ is even. If $k+l$ is odd, the integral is zero.
\end{lem}

\begin{proof}
When $k+l$ is even, we calculate, using (\ref{Geg1}) and (\ref{FH}),
\begin{align*}
&\int_{\mS^{m-1}}  C^{\l+1}_{k}(\la  \eta, \uy' \ra) C^{\l+1 }_{l}(\la  \eta, \ux' \ra) d \omega(\eta)\\
 &= \int_{\mS^{m-1}}  \left( \sum_{j=0}^{\lfloor \frac{k}{2}\rfloor} \frac{\l + k-2j}{\l} C^{\l}_{k-2j}(\la  \eta, \uy' \ra)\right) \left( \sum_{j=0}^{\lfloor \frac{l}{2}\rfloor} \frac{\l + l-2j}{\l} C^{\l}_{l-2j}(\la  \eta, \ux' \ra)\right) d \omega(\eta)\\
 &= \int_{\mS^{m-1}}  \left( \sum_{j=0}^{\lfloor \frac{k}{2}\rfloor} \left(\frac{\l + k-2j}{\l}\right)^{2} C^{\l}_{k-2j}(\la  \eta, \uy' \ra) C^{\l}_{k-2j}(\la  \eta, \ux' \ra)\right)  d \omega(\eta)\\
 &=\sum_{j=0}^{\lfloor \frac{k}{2}\rfloor} \frac{\l + k-2j}{\l} C^{\l}_{k-2j}(\la  \ux', \uy' \ra)\\
 &=C^{\l+1}_{k}(\la  \ux', \uy' \ra).
\end{align*}
If $k+l$ is odd, a similar calculation shows the result is zero.
\end{proof}

\begin{lem} \label{IntegralInnerInner}
Suppose $k \leq l$. Then one has
\begin{align*}
&\int_{\mS^{m-1}}  \la  \eta, \uy' \ra C^{\l+1}_{k}(\la  \eta, \uy' \ra)  \la  \eta, \ux' \ra C^{\l+1 }_{l}(\la  \eta, \ux' \ra) d \omega(\eta)\\
& =  \begin{cases} \la  \ux', \uy' \ra C^{\l+1}_{k}(\la  \ux', \uy' \ra)& k < l\\ \frac{k+1}{2(k+1+\l)} \la  \ux', \uy' \ra C^{\l+1}_{k}(\la  \ux', \uy' \ra) + \frac{k+1 + 2\l}{2(k+1+\l)} C^{\l+1}_{k-1}(\la  \ux', \uy' \ra) & k=l. \end{cases}  
\end{align*}
provided $k+l$ is even. If $k+l$ is odd, the integral is zero.
\end{lem}

\begin{proof}
Take $k+l$ even.
First decompose $\la  \eta, \uy' \ra C^{\l+1}_{k}(\la  \eta, \uy' \ra)$  and  $\la  \eta, \ux' \ra C^{\l+1 }_{l}(\la  \eta, \ux' \ra)$ according to the recursion formula (\ref{Geg2}). This reduces the integral to 4 different integrals that we calculate separately. We obtain
\begin{align*}
I_{1} &= \frac{(k+1)(l+1)}{4(k+\l+1)(l+\l+1)} \int_{\mS^{m-1}}  C^{\l+1}_{k+1}(\la  \eta, \uy' \ra) C^{\l+1 }_{l+1}(\la  \eta, \ux' \ra) d \omega(\eta)\\
&= \frac{(k+1)(l+1)}{4(k+\l+1)(l+\l+1)} C^{\l+1}_{k+1}(\la  \ux', \uy' \ra)
\end{align*}
using Lemma \ref{DimRedInt}. Similarly, we find
\begin{align*}
I_{2} &= \frac{(k+1)(l+1 +2\l)}{4(k+\l+1)(l+\l+1)} \int_{\mS^{m-1}}  C^{\l+1}_{k+1}(\la  \eta, \uy' \ra) C^{\l+1 }_{l-1}(\la  \eta, \ux' \ra) d \omega(\eta)\\
&= \frac{(k+1)(l+1+2\l)}{4(k+\l+1)(l+\l+1)} \left\{ \begin{array}{ll} C^{\l+1}_{k+1}(\la  \ux', \uy' \ra) & k < l \\  C^{\l+1}_{l-1}(\la  \ux', \uy' \ra)&  k = l \end{array} \right.
\end{align*}
and also
\begin{align*}
I_{3} &= \frac{(k+1+2\l)(l+1)}{4(k+\l+1)(l+\l+1)} \int_{\mS^{m-1}}  C^{\l+1}_{k-1}(\la  \eta, \uy' \ra) C^{\l+1 }_{l+1}(\la  \eta, \ux' \ra) d \omega(\eta)\\
&= \frac{(k+1+2\l)(l+1)}{4(k+\l+1)(l+\l+1)} C^{\l+1}_{k-1}(\la  \ux', \uy' \ra)\\
I_{4} &= \frac{(k+1+2\l)(l+1+2\l)}{4(k+\l+1)(l+\l+1)} \int_{\mS^{m-1}}  C^{\l+1}_{k-1}(\la  \eta, \uy' \ra) C^{\l+1 }_{l-1}(\la  \eta, \ux' \ra) d \omega(\eta)\\
&= \frac{(k+1+2\l)(l+1+2\l)}{4(k+\l+1)(l+\l+1)} C^{\l+1}_{k-1}(\la  \ux', \uy' \ra).
\end{align*}
Now summing $I_{1} +I_{2} + I_{3} +I_{4}$ and using formula (\ref{Geg2}) in the other direction completes the proof.
\end{proof}

\begin{lem} \label{IntegralInnerWedge}
Put $I_{k} = \{k, k-2, k-4, \ldots \}$. Then 
\begin{align*}
& \int_{\mS^{m-1}} \la  \eta, \ux' \ra  C^{\l+1}_{k}(\la  \eta, \ux' \ra) (\eta \wedge \uy') C^{\l+1}_{l}(\la  \eta, \uy' \ra) d \omega(\eta)\\
& \qquad\qquad   = \begin{cases} 0& l \not \in I_{k}\\
(\ux' \wedge \uy') C^{\l+1}_{l}(\la  \ux', \uy' \ra) &l+2 \in I_{k}\\ 
\frac{k+1}{2(k+\l+1)}(\ux' \wedge \uy') C^{\l+1}_{l}(\la  \ux', \uy' \ra) & l=k.
 \end{cases}
\end{align*}
\end{lem}

\begin{proof}
First expand $\la  \eta, \ux' \ra  C^{\l+1}_{k}(\la  \eta, \ux' \ra) $ using (\ref{Geg2}) and apply (\ref{Geg1}) to the result. The lemma then follows immediately using Lemma \ref{IntegralWedge}.
\end{proof}

We need some algebraic results before we can prove the final lemma. Recall that  
\begin{align*}
\eta \wedge \ux' &= -\ux' \eta - \la \eta, \ux' \ra,\\
\uy' \wedge \eta &= -\eta \uy' - \la \eta, \uy' \ra.
\end{align*}
This allows us to compute
\begin{align*}
(\eta \wedge \ux' )( \uy' \wedge \eta ) &= (\ux' \eta + \la \eta, \ux' \ra) (\eta \uy' + \la \eta, \uy' \ra)\\
&= \la \eta, \ux' \ra \la \eta, \uy' \ra - \ux' \uy' |\eta|^{2} + \ux' \eta  \la \eta, \uy' \ra + \eta \uy'  \la \eta, \ux' \ra\\
&=  |\eta|^{2} \la \ux', \uy' \ra - (\ux' \wedge \uy') |\eta|^{2} + (\ux' \wedge \eta)   \la \eta, \uy' \ra\\
& - \la \ux', \eta \ra   \la \eta, \uy' \ra + (\eta \wedge \uy')  \la \eta, \ux' \ra\\ 
&=\la \ux', \uy' \ra - (\ux' \wedge \uy')  + (\ux' \wedge \eta)   \la \eta, \uy' \ra - \la \ux', \eta \ra   \la \eta, \uy' \ra + (\eta \wedge \uy')  \la \eta, \ux' \ra, 
\end{align*}
where we used $|\eta|^{2}=1$.
This decomposition allows us to obtain the following lemma.

\begin{lem}\label{IntegralWedgeWedge}
One has
\begin{align*}
&\int_{\mS^{m-1}} (\eta \wedge \ux')  C^{\l+1}_{k}(\la  \eta, \ux' \ra) ( \uy' \wedge \eta) C^{\l+1}_{l}(\la  \eta, \uy' \ra) d \omega(\eta)\\
& = \delta_{k l} \frac{(k+1)(k+1+2\l)}{4\l(k+\l+1)}  C^{\l}_{k+1}(\la \ux', \uy' \ra ) - \delta_{k l} \frac{\l}{k+\l+1} (\ux' \wedge \uy') C^{\l+1}_{k}(\la \ux', \uy' \ra ).
\end{align*}
\end{lem}

\begin{proof}
First use the formula 
\begin{align*}
(\eta \wedge \ux' )( \uy' \wedge \eta ) &= \la \ux', \uy' \ra - (\ux' \wedge \uy')  + (\ux' \wedge \eta)   \la \eta, \uy' \ra\\
& - \la \ux', \eta \ra   \la \eta, \uy' \ra + (\eta \wedge \uy')  \la \eta, \ux' \ra.
\end{align*}
This splits the integral in 5 terms. These terms can immediately be calculated using the Lemmas \ref{DimRedInt}, \ref{IntegralInnerInner} and \ref{IntegralInnerWedge}. Putting everything together yields 
\begin{align*}
&\int_{\mS^{m-1}} (\eta \wedge \ux')  C^{\l+1}_{k}(\la  \eta, \ux' \ra) ( \uy' \wedge \eta) C^{\l+1}_{l}(\la  \eta, \uy' \ra) d \omega(\eta)\\
& = \delta_{k l} \frac{k+1+2\l}{2(k+\l+1)}  \left( \la \ux', \uy' \ra C^{\l+1}_{k}(\la \ux', \uy' \ra ) -  C^{\l+1}_{k-1}(\la \ux', \uy' \ra ) \right)\\
& - \delta_{k l} \frac{\l}{k+\l+1} (\ux' \wedge \uy') C^{\l+1}_{k}(\la \ux', \uy' \ra ).
\end{align*}
The first term can be further simplified as follows 
\begin{align*}
w C^{\l+1}_{k}(w ) -  C^{\l+1}_{k-1}(w) &= \frac{k+1}{2(k+\l+1)} C^{\l+1}_{k+1}(w ) + \left( \frac{k+1+2\l}{2(k+\l+1)} -1\right) C^{\l+1}_{k-1}(w )\\
&= \frac{k+1}{2(k+\l+1)} \left( C^{\l+1}_{k+1}(w ) -  C^{\l+1}_{k-1}(w )  \right)\\
&= \frac{k+1}{2\l}C^{\l}_{k+1}(w ),
\end{align*}
where we subsequently used (\ref{Geg2}) and (\ref{Geg1}). This completes the proof.
\end{proof}

Now we have all necessary ingredients to establish the key step.

\begin{thm}
For all $m \in \NN_0$, $m \ge 2$, 
\[
\int_{\mS^{m-1}} \overline{K_-( r \eta,x)} K_-(y,r \eta) d\omega(\eta)  = 2^{\l}\Gamma(\l+1) u^{-\l} J_{\lambda}(u)
\]
with $u = r \sqrt{|\ux|^{2}+|\uy|^{2} - 2\la \ux,\uy\ra}$ and $\l = (m-2)/2$.
\label{IntegralKernelKernel}
\end{thm}

\begin{proof}
First we rewrite $K_-(x,y)$ (see Theorem \ref{seriesthm}) as 
\[
K_-(x,y) = F_{\l}(w,z) + \ux'\wedge \uy' G_{\l}(w,z)
\]
with
\begin{align*}
F_{\l}(w,z) &= 2^{\l-1} \Gamma(\l) \sum_{k=0}^{\infty} f_{k} z^{-\l} J_{k+\l}(z) C^{\l}_{k}(w),\\
G_{\l}(w,z) &= 2^{\l-1} \Gamma(\l) \sum_{k=1}^{\infty} g_{k} z^{-\l} J_{k+\l}(z) C^{\l+1}_{k-1}(w)
\end{align*}
and with $\ux' = \ux/|\ux|$, $\uy' = \uy/|\uy|$, $w = \la \ux',\uy' \ra$, $z = |\ux||\uy|$ and $\l = (m-2)/2$. The coefficients $f_{k}$ and $g_{k}$ are given by
\begin{align*}
f_{k}&= \l (i^{m}+(-1)^{k}) - (k+\l)  (i^{m}-(-1)^{k}) \\
g_{k}&= -2\l (i^{m}+(-1)^{k}). 
\end{align*}
Using this decomposition of $K_-$, the integral
\[
\int_{\mS^{m-1}} \overline{K_-( r \eta,x)} K_-(y,r \eta) d\omega(\eta) 
\]
splits in 4 pieces $I_{1}+I_{2}+I_{3}+I_{4}$ which we calculate separately. For $I_{1}$, we can use (\ref{FH}) to obtain
\begin{align*}
I_{1} &= \int_{\mS^{m-1}} \overline{F_{\l}(z_{1},w_{1})}F_{\l}(z_{2},w_{2}) d\omega(\eta)\\
&= 4^{\l-1} \Gamma(\l)^{2} \sum_{k=0}^{\infty} \frac{\l}{\l+k}\overline{f_{k}} f_{k} (z_{1}z_{2})^{-\l} J_{k+\l}(z_{1})J_{k+\l}(z_{2}) C^{\l}_{k}(\la \ux' , \uy' \ra)
\end{align*}
where we use the notations $z_{1} = r |\ux|$, $z_{2}=r |\uy|$, $w_{1} = \la \eta, \ux' \ra$ and $w_{2} = \la \eta, \uy' \ra$.
For $I_{2}$, we use Lemma \ref{IntegralWedge} to see that the non-diagonal terms again vanish, yielding
\begin{align*}
I_{2} &= \int_{\mS^{m-1}} \eta \wedge \ux' \overline{G_{\l}(z_{1},w_{1})}F_{\l}(z_{2},w_{2}) d\omega(\eta)\\
&= -4^{\l-1} \Gamma(\l)^{2} \ux'\wedge \uy' \sum_{k=1}^{\infty} \frac{\l}{\l+k}\overline{g_{k}} f_{k} (z_{1}z_{2})^{-\l} J_{k+\l}(z_{1})J_{k+\l}(z_{2}) C^{\l+1}_{k-1}(\la \ux' , \uy' \ra)
\end{align*}
and similarly for $I_{3}$
\begin{align*}
I_{3} &= \int_{\mS^{m-1}} \overline{F_{\l}(z_{1},w_{1})}  \uy' \wedge \eta G_{\l}(z_{2},w_{2}) d\omega(\eta)\\
&= -4^{\l-1} \Gamma(\l)^{2} \ux'\wedge \uy' \sum_{k=1}^{\infty} \frac{\l}{\l+k}\overline{f_{k}} g_{k} (z_{1}z_{2})^{-\l} J_{k+\l}(z_{1})J_{k+\l}(z_{2}) C^{\l+1}_{k-1}(\la \ux' , \uy' \ra).
\end{align*}
Finally, we can calculate the term $I_{4}$ using Lemma \ref{IntegralWedgeWedge} as follows
\begin{align*}
I_{4} &= \int_{\mS^{m-1}} \eta \wedge \ux' \overline{G_{\l}(z_{1},w_{1})} \uy' \wedge \eta G_{\l}(z_{2},w_{2}) d\omega(\eta)\\
&= 4^{\l-1} \Gamma(\l)^{2} \sum_{k=1}^{\infty} \frac{k(k+2\l)}{4\l(k+\l)}\overline{g_{k}} g_{k} (z_{1}z_{2})^{-\l} J_{k+\l}(z_{1})J_{k+\l}(z_{2}) C^{\l}_{k}(\la \ux' , \uy' \ra)\\
& - 4^{\l-1} \Gamma(\l)^{2} \ux'\wedge \uy' \sum_{k=1}^{\infty} \frac{\l}{k+\l}\overline{g_{k}} g_{k} (z_{1}z_{2})^{-\l} J_{k+\l}(z_{1})J_{k+\l}(z_{2}) C^{\l+1}_{k-1}(\la \ux' , \uy' \ra).
\end{align*}
Adding these 4 terms then gives
\begin{align*}
&I_{1}+I_{2}+I_{3}+I_{4}\\
 &=4^{\l-1} \Gamma(\l)^{2} \sum_{k=0}^{\infty} \left(\frac{\l}{\l+k}\overline{f_{k}} f_{k} + \frac{k(k+2\l)}{4\l(k+\l)}\overline{g_{k}} g_{k} \right)(z_{1}z_{2})^{-\l}\\
&\qquad \times  J_{k+\l}(z_{1})J_{k+\l}(z_{2}) C^{\l}_{k}(\la \ux' , \uy' \ra) \\
&-4^{\l-1} \Gamma(\l)^{2} \ux'\wedge \uy' \sum_{k=1}^{\infty} \frac{\l}{k+\l}\left(\overline{g_{k}} f_{k}  +\overline{f_{k}} g_{k} + \overline{g_{k}} g_{k} \right) (z_{1}z_{2})^{-\l}\\
&\qquad\times J_{k+\l}(z_{1})J_{k+\l}(z_{2}) C^{\l+1}_{k-1}(\la \ux' , \uy' \ra).
\end{align*}
It is not difficult to check for all $k$ that $\overline{g_{k}} f_{k}  +\overline{f_{k}} g_{k} + \overline{g_{k}} g_{k} =0$, so the term in $\ux'\wedge \uy'$ vanishes. Similarly, we can compute that
\[
\frac{\l}{\l+k}\overline{f_{k}} f_{k} + \frac{k(k+2\l)}{4\l(k+\l)}\overline{g_{k}} g_{k} = 4\l (k+\l)
\]
for $k>0$. For $k=0$ we only have the term $\overline{f_{0}} f_{0} = 4 \l^{2}$. This allows to conclude that
\begin{align*}
I_{1}+I_{2}+I_{3}+I_{4} = 4^{\l} \l \Gamma(\l)^{2} \sum_{k=0}^{\infty}  (k+\l)(z_{1}z_{2})^{-\l} J_{k+\l}(z_{1})J_{k+\l}(z_{2}) C^{\l}_{k}(\la \ux' , \uy' \ra).
\end{align*}
Now we invoke the addition formula for Bessel functions (see \cite{MR0010746}, section 11.4), given by
\[
u^{-\l} J_{\l}(u)= 2^{\l} \Gamma(\l)\sum_{k=0}^{\infty}  (k+\l)(z_{1}z_{2})^{-\l} J_{k+\l}(z_{1})J_{k+\l}(z_{2}) C^{\l}_{k}(\la \ux' , \uy' \ra)
\]
with $u = r \sqrt{|\ux|^{2}+|\uy|^{2} - 2\la \ux,\uy\ra}$. This completes the proof of the theorem.
\end{proof}

We can now prove our main theorem in this section. 

\medskip\noindent
{\it Proof of Theorem \ref{transl}.}
If $f(x)=f_{0}(|\ux|)$ is real-valued and radial, then $\cF_{-}(f_{0})(x)$ is a radial function as well and it coincides with the ordinary Fourier transform $\wh f(x)$. With a slight abuse of notation, we write $\cF_{-}(f_{0})(x) = \wh f_0(r)$, with $r = |\ux|$. Using polar coordinates $\xi = r \eta$, $r = |\xi|$, we can then write
\[
\tau_{y} f (x) = \frac{2^{1-\frac{m}{2}}}{\Gamma(\frac{m}{2})} \int_{0}^{+\infty} r^{m-1}\wh f_{0}(r)  \left[ \int_{\mS^{m-1}} \overline{K_-( r \eta,x)} K_-(y,r \eta) d\omega(\eta)  \right] \, dr. 
\]
By Theorem \ref{IntegralKernelKernel}, we obtain 
\begin{align*}
\tau_{y} f (x)  = \int_{0}^{+\infty} r^{m-1} \widehat{f}(r) u^{-\l} J_{\lambda}(u) dr
\end{align*}
with  $u = r \sqrt{|\ux|^{2}+|\uy|^{2} - 2\la \ux,\uy\ra}$. This means that $\tau_y f$ is the Hankel transform of $\wh f$ by \eqref{Hankl}. However, $\wh f(x) = H_\l f_0(|\ux|)$, so that by the inversion of the Hankel transform, we obtain 
\[
\tau_{y} f(x)= f_{0}\left(\sqrt{|\ux|^{2}+|\uy|^{2} - 2\la \ux,\uy\ra}\right) = f_0(|\ux - \uy|),
\]
thus completing the proof.
\qed \medskip

\section{Generalized convolution and Clifford-Fourier transform}
\label{SecConvolution}
\setcounter{equation}{0}

Using the generalized translation, we can define a convolution for functions with values in Clifford algebra. 

\begin{defn}
For $f, g \in \cS(\RR^m) \otimes \cC l_{0,m}$, the generalized convolution, $f *_{Cl} g$, is defined by 
$$
    (f *_{Cl} g)(x) := (2\pi)^{-m/2}\int_{\RR^m} \tau_y f(x) g(y) dy, \qquad x \in \RR^m. 
$$  
\end{defn}

If $f$ and $g$ both take values in the Clifford algebra, then $f *_{Cl} g$ is not commutative in general. 
We are interested in the case when one of the two functions is radial. 
  
\begin{thm} 
If $ g\in \cS(\RR^m) \otimes \cC l_{0,m}$ and $f \in \cS(\RR^m)$ is a radial function, then $f *_{Cl} g$ satisfies 
$$
   \CF_- ( f *_{Cl} g)(x) = \CF_- f (x)   \CF_-g(x).
$$
In particular, under these assumptions one has
\[
f *_{Cl} g = g *_{Cl} f.
\]
\end{thm}

\begin{proof}
If $f$ is a radial function then so is $\CF_- f$, which implies in particular that $\CF_-f \cdot h = h \cdot\CF_-f$ for any Clifford algebra valued function $h$. Hence, by the definition of the Clifford-Fourier transform and the Fubini theorem, 
\begin{align*}
    \CF_- (f*g)(x) & = (2\pi)^{-m} \int_{\RR^m} K_-(\xi,x) \left[ \int_{\RR^m} \tau_y f(\xi) g(y) dy \right] d \xi\\ 
    & = (2\pi)^{-m/2}\int_{\RR^m} \left[ (2\pi)^{-m/2}\int_{\RR^m}   K_-(\xi,x) \tau_y f(\xi) d \xi  \right] g(y) dy  \\
    & = (2\pi)^{-m/2}\int_{\RR^m} \CF_- (\tau_y f)(x) g(y) dy    = (2\pi)^{-m/2} \int_{\RR^m} K_-(y,x) \CF_- f (x) g(y) dy \\
    & = (2\pi)^{-m/2}  \CF_- f (x)  \int_{\RR^m} K_-(y,x)  g(y) dy = \CF_- f (x)   \CF_-g(x)
\end{align*}
where we have used the fact that $K_-(y,x) \CF_- f (x) = \CF_- f (x)  K_-(y,x)$. 
The same proof shows also that
\[
\CF_- ( g *_{Cl} f)(x) = \CF_- g (x)   \CF_- f(x)
\]
from which it immediately follows that $f *_{Cl} g = g *_{Cl} f$.
\end{proof}

Since the Clifford-Fourier transform coincides with the ordinary Fourier transform for radial
functions, we have in particular that, for $\phi(x) =  e^{-|\ux|^2/2}$, 
$$
          \CF_-  \phi (x) = \wh \phi(x) =  \phi(x).  
$$
We denote by $\phi_t$ the function $\phi_t(x):=   t^{-m/2} \phi(x/\sqrt{t})$, 
$t > 0$. A change of variable shows $\phi( \sqrt{t} x) = \wh \phi_t (x)$. 

\begin{lem} 
Let $m$ be even. If $f \in B(\RR^m) \otimes \cC l_{0,m}$, then
$$
    \phi_t  *_{Cl}f (x)= (2\pi)^{-m/2} \int_{\RR^m} K_-(\xi,x) \CF_- f(\xi) \phi (\sqrt{t} \xi)d \xi, \qquad  x \in \RR^m. 
$$
\end{lem}

\begin{proof}
Just like in the classical case, this is a simple application of Fubini' s theorem.  
\begin{align*}
    (  \phi_t  *_{Cl}f ) (x) & = (2\pi)^{-m/2} \int_{\RR^m} \tau_y \phi_t(x) f(y) dy  \\
       & = (2\pi)^{-m}  \int_{\RR^m} \left[ \int_{\RR^m} K_-(\xi, x) K_-(y, \xi)  \wh \phi_t (\xi) d\xi \right] f(y)dy \\
       & = (2\pi)^{-m/2} \int_{\RR^m} K_-(\xi, x) \left[ (2\pi)^{-m/2}\int_{\RR^m} K_-(y, \xi)  f(y)dy \right] \phi (\sqrt{t} \xi) d\xi  \\
       & =  (2\pi)^{-m/2}\int_{\RR^m} K_-(\xi, x) \CF_- f (\xi) \phi (\sqrt{t} \xi) d\xi, 
\end{align*}
where we have used the fact that $\wh \phi_t (\xi) = \phi( \sqrt{t} \xi)$ is a radial function so that it 
commutes with $f(y)$. 
\end{proof}

We are now in the position to establish an inversion formula for the Clifford-Fourier transform in $B(\RR^m) \otimes \cC l_{0,m}$.

\begin{thm} 
Let $m$ be even. If $f \in B(\RR^m) \otimes \cC l_{0,m}$ and $ \CF_- f \in B(\RR^m) \otimes \cC l_{0,m}$, and if 
$$
   g(x) = (2\pi)^{-m/2} \int_{\RR^m}  K_-(\xi,x) \CF_- f(\xi) d \xi, \qquad  x \in \RR^m, 
$$
then $g \in C(\RR^m)$ and $f(x) = g(x)$ a.e.
\end{thm}

\begin{proof}
By the above lemma, 
$$
   (\phi_t  *_{Cl} f)(x) = (2\pi)^{-m/2} \int_{\RR^m}  K_-(\xi,x) \CF_- f(\xi) e^{- t |\xi|^2/2} d \xi. 
$$
For each $x$, the bound of the kernel shows that the integrand in the right-hand side is bounded by 
$c(x) (1+ |\xi|)^{(m-2)/2}\CF_- f(\xi)$ with $c(x) = (1+|\ux|)^{(m-2)/2}$, which has a finite 
integral as $\CF_- f \in B(\RR^m) \otimes \cC l_{0,m}$,  so that the right side converges, as $t \to 0$, to $g(x)$ for 
every $x \in \RR^m$ by the dominated convergence theorem. 

On the other hand, since $\phi_t$ is radial, $\phi_t  *_{Cl} f$ coincides with the classical 
convolution, $\phi_t  *_{Cl} f = \phi_t*f$. It is well-known (cf. \cite{SW}) that if
$\phi_t*f$ converges to $f$ in norm, there is a subsequence $t_j$ so that $ \phi_{t_j}*f$ 
converges to $f$ a.e., so that $f(x) = g(x)$ a.e. 
\end{proof}

\end{document}